\documentclass[english,8pt,reqno]{amsart}
\usepackage{geometry}  
\usepackage{amssymb,amsmath,amsthm,amsfonts,color}
\usepackage{mathrsfs,dsfont, comment,mathscinet}
\usepackage{graphicx}
\usepackage{epstopdf}
\usepackage{mathtools}
\usepackage{babel}
\usepackage{enumerate,esint}

\usepackage{indentfirst}
\usepackage{bm}
\usepackage{picinpar}
\usepackage{caption}
\usepackage{subcaption}

\usepackage[colorlinks=true, pdfstartview=FitV, linkcolor=blue, citecolor=blue, urlcolor=blue]{hyperref}

\usepackage[toc,page]{appendix}

\usepackage{etoolbox}
\usepackage{authblk}
\usepackage{amsaddr}

\usepackage[running]{lineno}

\newcommand*\patchAmsMathEnvironmentForLineno[1]{%
	\expandafter\let\csname old#1\expandafter\endcsname\csname #1\endcsname
	\expandafter\let\csname oldend#1\expandafter\endcsname\csname end#1\endcsname
	\renewenvironment{#1}%
	{\linenomath\csname old#1\endcsname}%
	{\csname oldend#1\endcsname\endlinenomath}}%
\newcommand*\patchBothAmsMathEnvironmentsForLineno[1]{%
	\patchAmsMathEnvironmentForLineno{#1}%
	\patchAmsMathEnvironmentForLineno{#1*}}%
\AtBeginDocument{%
	\patchBothAmsMathEnvironmentsForLineno{equation}%
	\patchBothAmsMathEnvironmentsForLineno{align}%
	\patchBothAmsMathEnvironmentsForLineno{flalign}%
	\patchBothAmsMathEnvironmentsForLineno{alignat}%
	\patchBothAmsMathEnvironmentsForLineno{gather}%
	\patchBothAmsMathEnvironmentsForLineno{multline}%
}

\graphicspath{ {images/} }

\allowdisplaybreaks 

\mathtoolsset{showonlyrefs} 

\makeatletter

\usepackage{fancyhdr}

\makeatletter
\patchcmd{\@maketitle}
{\ifx\@empty\@dedicatory}
{\ifx\@empty\@date \else {\vskip3ex \centering\footnotesize\@date\par\vskip1ex}\fi
	\ifx\@empty\@dedicatory}
{}{}
\patchcmd{\@maketitle}
{\ifx\@empty\@date\else \@footnotetext{\@setdate}\fi}
{}{}{}
\makeatother

\begingroup
\newtheorem{theorem}{Theorem}[section]
\newtheorem{lemma}[theorem]{Lemma}
\newtheorem{proposition}[theorem]{Proposition}
\newtheorem{corollary}[theorem]{Corollary}

\endgroup

\theoremstyle{definition}
\begingroup
\newtheorem{define}[theorem]{Definition}
\newtheorem{remark}[theorem]{Remark}

\endgroup

\newcommand\ba[1]{\begin{align}\label{#1}}
	\newcommand\ea{\end{align}}
\newcommand\bas{\begin{align*}}
	\newcommand\eas{\end{align*}}

\newcommand\ee{\end{equation}}
\newcommand\be{\begin{equation}}
\newcommand\ees{\end{equation*}}
\newcommand\bes{\begin{equation*}}

\mathchardef\emptyset="001F

\mathchardef\emptyset="001F

\newcommand{\e}{\varepsilon}

\newcommand{\Om}{\Omega}

\newcommand{\R}{{\mathbb R}}

\newcommand{\rn}{{{\R}^N}}

\newcommand{\wto}{\rightharpoonup}
\newcommand{\wtos}{\mathrel{\mathop{\rightharpoonup}\limits^*}}

\newcommand{\N}{{\mathbb{N}}}

\newcommand{\wtogs}{\mathrel{\mathop{\to}\limits^{\text {s-s}}}}
\newcommand{\wtogsp}{\mathrel{\mathop{\to}\limits^{\text{s-s }(p)}}}

\newcommand\norm[1]{\left\|#1\right\|}

\newcommand{\abs}[1]{\left\lvert#1\right\rvert} 
\newcommand{\fsp}[1]{\left(#1\right)} 
\newcommand{\fmp}[1]{\left[#1\right]}
\newcommand{\flp}[1]{\left\{#1\right\}}

\newcommand{\vp}{\varphi}

\newcommand{\limn}{\lim_{n\rightarrow\infty}}

\newcommand{\lime}{\lim_{\e\rightarrow0}}

\newcommand{\pd}{{\partial}}
\newcommand{\divg}{{\operatorname{div}}}

\newcommand{\seqn}[1]{\left\{#1\right\}}

\newcommand{\liminfn}{{\liminf_{n\to\infty}}}
\newcommand{\ir}{{\lfloor r\rfloor}}

\newcommand{\IM}{{\operatorname{IM}}}

\newcommand\M{\mathbb M}

\definecolor{CMUred}{RGB}{153,0,0}
\definecolor{CMUgreen}{RGB}{0,135,81}
\definecolor{CMUblue}{RGB}{0,51,127}
\definecolor{Pblue}{RGB}{87,158,208}

\newcommand{\argmin}{{\operatorname{arg\,min}}}

\newcommand{\hnmo}{{\mathcal H^{N-1}}}

\newcommand\mb{{\mathcal{M}_b}}

\newcommand{\magenta}{ \color{black} }

\def\argmin{\mathop{\rm arg\, min}}

\DeclareRobustCommand{\Om}{\Omega}

\numberwithin{equation}{section}

\newcommand{\normmm}[1]{{\left\vert\kern-0.25ex\left\vert\kern-0.25ex\left\vert #1 
		\right\vert\kern-0.25ex\right\vert\kern-0.25ex\right\vert}}

\newcommand{\BLK}{\color{black}}

\definecolor{CMUred}{RGB}{153,0,0}
\definecolor{CMUgreen}{RGB}{0,135,81}
\definecolor{CMUblue}{RGB}{0,51,127}
\definecolor{Pblue}{RGB}{87,158,208}

\newcommand{\ellp}{{\ell^p}}

\def\argmin{\mathop{\rm arg\, min}}

\DeclareRobustCommand{\Om}{\Omega}

\numberwithin{equation}{section}

\title{Real order total variation with applications to the loss functions in learning schemes}

\author[P. Liu] {Pan Liu}
\address[Pan Liu]{Medical Big Data Research Center, \\
	Chinese PLA General Hospital,\\
	Beijing 100853, China}
\email[P. Liu] {dragonrider.liupan@gmail.com}

\author[X.Y. Lu] {Xin Yang Lu}
\address[Xin Yang Lu]{Department of Mathematical Sciences,\\
	 Lakehead University,
	955 Oliver Road, Thunder Bay, ON, Canada}
\email[X. Lu] {xlu8@lakeheadu.ca}

\author[K. He]{Kunlun He}
\address[Kunlun He]{(corresponding author)\\
	Medical Big Data Research Center,\\
	Chinese PLA General Hospital,\\
	Beijing 100853, China}
\email[K. He]{kunlunhe@plagh.org}


\date{}                                           

\begin{document}
	
\begin{abstract}
	 Loss function are an essential part in modern data-driven approach, such as bi-level training scheme and machine learnings.
	 In this paper we propose a loss function
	 consisting of a $r$-order (an)-isotropic total variation semi-norms $TV^r$,
	 $r\in \R^+$, defined via the Riemann-Liouville (R-L) fractional derivative. We focus on studying key theoretical properties, such as the lower semi-continuity and compactness with respect to both the function
	and the order of derivative $r$, of such loss functions.
\end{abstract}

\maketitle
\tableofcontents

\thispagestyle{empty}

\section{Introduction}\label{sec:intro}

Medical image processing is a crucial step in routine clinical diagnoses, which includes tasks such as image denoising,
image segmentation and reconstruction. Methods for such tasks can be roughly classified into two classes, based on the approach: one is the class of structure-driven
approaches such as variational method proposed in \cite{rudin1992nonlinear}; the other is the class of data-driven approaches such as bi-level learning scheme and deep learning method \cite{DBLP:conf/nips/KrizhevskySH12,DBLP:journals/corr/SimonyanZ14a,DBLP:conf/cvpr/SzegedyLJSRAEVR15,DBLP:conf/cvpr/HeZRS16,DBLP:conf/cvpr/HuangLMW17,DBLP:conf/nips/VaswaniSPUJGKP17} which becomes more popular in recent image processing industry.\\

We take image denoising task as example: let $u_\eta$ be a corrupted image, one of the classical structure-driven approach, the \emph{ROF} model
(see \cite{rudin1992nonlinear})
	shown that the task of computing
	a denoised, or reconstructed, image from a noisy image $u_\eta$ is equivalent to solving the problem
\be\label{min_rof}
u_\alpha := \argmin\flp{\norm{u-u_\eta}_{L^2(Q)}^2+\alpha TV(u),\,\, u\in BV(Q)},
\ee
where $Q:=(0,1)\times(0,1)$ denotes a squared image and $BV(Q):=\flp{u\in L^1(Q):\,\, TV(u)<+\infty}$, where 
\be
 TV(u):=\sup\flp{\int_Q u \text{ div} \varphi \,dx:\varphi\in C_c^\infty(Q)}
 \ee
 is the total variation semi-norm defined in \cite{evans2015measure}.\\\\
Then, the minimizer $u_{\alpha}$ in \eqref{min_rof} is the denoised image for a given $\alpha>0$. The advantage for such structure-driven method is it does not depends on image itself, robust so that for any given corrupted image $u_\eta$, we can certainly obtain a reasonable solution $u_\alpha$, however, the disadvantage is, without {\magenta any} additional supervision,
choosing $\alpha$ too large would often result in the loss of relevant image details, and choosing $\alpha$ too small would produce an image with too much noise. \\

The data-driven approach, on the other hand, could incorporate the given prior knowledge of the problem in terms of a training set. 
In image denoising,
a ``training set'', or a ``dataset'' $\mathcal D$, is a finite family of pair of images
\be\label{dataset}
\mathcal D = \flp{(u_c, u_\eta),\,\,\flp{\text{finitely many of pairs}}},
\ee
containing
a corrupted one $u_\eta$, and the corresponding clean data $u_c$.
A typical data-driven approach can be defined as follows:
\be\label{learning_scheme}
\tilde\omega \in \argmin\flp{\sum_{u_c\in\mathcal D}\mathcal{L}\fsp{u_\omega-u_c}:\,\,\omega\in W},
\ee
for 
\be\label{neural_network}
u_\omega = \mathcal N_\omega(u_\eta).
\ee
The functional $\mathcal{N}_\omega$ in \eqref{neural_network} can be choosing from a widely range of options. For instance,
$\mathcal{N}_\omega$ could be the functional from the equation \eqref{min_rof} where we set $\omega=\alpha\in\R^+=W$. Other possible choices for $\mathcal N_\omega$ can be, 
for example,
a convolutional neural network such as AlexNet \cite{DBLP:conf/nips/KrizhevskySH12}, VGG \cite{DBLP:journals/corr/SimonyanZ14a}, GoogleNet \cite{DBLP:conf/cvpr/SzegedyLJSRAEVR15}, ResNet \cite{DBLP:conf/cvpr/HeZRS16}, DenseNet \cite{DBLP:conf/cvpr/HuangLMW17}, and Transformers \cite{DBLP:conf/nips/VaswaniSPUJGKP17}, in which the $\omega$ is a set of matrix with real numbers. The purpose of this article is, however, not to study neural network, and hence we will not delve into the regularity issues of such functional $\mathcal{N}$ and leave the interested readers to the reference mentioned above.\\\\
The main focus of this article is the \emph{loss function} $\mathcal L$, another critical part in data-driven approach which measures the difference between the clean image $u_c$ and the reconstructed image $u_\omega$. The choice of loss function is a delicate and crucial one, as it decides how \eqref{learning_scheme} to optimizes $\omega$. Popular choice of loss functions are $L^2$ or $L^1$ norm (also called MSE and MAE, respectively). However, such loss functions usually resulting in images that have high peak signal-to-noise ratios, but often lacking high-frequency details, and are therefore error-prone with respect to fine-scale, possibly clinically relevant details. 
Moreover, for medical images, especially the computed tomography angiography (CTA) scans for vessel reconstruction, the images are 3d dimensional and rather large in size. Therefore, the region of interested (ROI) constitutes a very small portion of the whole area, and hence, a suitable loss function must be able to accurately focus on both local area information, as well as global area information.   \\\\
In \cite{10.1007/978-3-030-32248-9_13}, a loss function composed with real order derivative was first introduced and successfully applied in GAN \cite{DBLP:journals/corr/GoodfellowPMXWOCB14}, a generative adversarial network, and received improved fine-scale details in reconstruction images. The loss function is concisely defined as the $r$-order total variation $TV^r$ (see Definition \ref{isotropic_frac_variation}), with $r\in(0,1)$ fixed. However, due to the limited literature at that time, no theoretical statements, such as compactness and semi continuity results, could be provided. \\\\
In this paper, we shall provide a full analysis of such real order derivative loss functions, and study our new functional $TV^r$, i.e. the \emph{real order} total variation semi-norm (of order $r$).

\begin{define}[Definition \ref{isotropic_frac_variation}]\label{TV_r_intro}
	Let $\ir$ denote the integer part of $r\in\R^+$. We define the \emph{$r$-order total variation} $TV^r(u)$ on $u\in L^1(Q)$ as follows. 
	\begin{enumerate}[1.]
		\item
		For $r=s\in(0,1)$ (i.e. $\ir=0$), we define
		\be\label{RLFOD_intro}
		TV^s(u):=\sup\flp{\int_Qu\, \divg^s \vp \,dx:\,\,\vp\in C_c^\infty(Q;\rn)\text{ and }\abs{\vp}\leq 1},
		\ee  
		where 
		$\divg^s$ is the ``fractional divergence'' operator,
		defined in \eqref{scaled_divg} below.
		\item
		For $r=\ir+s$ where $\ir\geq1$, we define 
		\be\label{RLFODr_intro}
		TV^r(u):=\sup\flp{\int_Qu\, \divg^s[\divg^\ir \vp] \,dx:\,\,\vp\in C_c^\infty(Q;\M^{N\times (N^{k})})\text{ and }\abs{\vp}\leq 1}.
		\ee
	\end{enumerate}
\end{define}
 The $TV^r$ can be thus considered as an extension of the classic
total variation semi-norm, which corresponds to the case $r=1$, as well as the $L^1$ norm, which corresponds to the case $r=0$.\\\\
Similarly, we introduce the functional space $BV^r$, defined as the space of
$L^1$ functions with bounded $TV^r$ semi-norm:
\begin{define}\label{strict_conv_BVr_one}
	Let $r\in\R^+$ be given. 
	\begin{enumerate}[1.]
		
		\item
		We define the (standard) $BV^r(Q)$ space by 
		\be\label{BV_r_p_diff_one}
		BV^r(Q):=\flp{u\in L^1(Q):\,\, TV^r(u)<+\infty}.
		\ee
		
		\item
		Given a sequence $\seqn{u_n}\subseteq L^1(Q)$ such that $TV^r(u_n)<+\infty$ for each $n\in\N$, we say it is \emph{strictly} converging to $u\in L^1(Q)$ 
		with respect to the $TV^r$ semi-norm, and write $u_n\wtogs u$, if
		\be\label{def_eq_strict_covg_one}
		\limn\norm{u-u_n}_{L^1(Q)}+\abs{TV^r(u_n)-TV^r(u)}= 0.
		\ee
		That is, $u_n\wtogs u$ if $u_n\to u$ strongly in $L^1(Q)$ and $TV^r(u_n)\to TV^r(u)$.

	\end{enumerate}
\end{define}
Thus, the space $BV^r$ can be therefore considered as an extension of the space $BV$.
The aim of this article is twofold. First, in Sections
\ref{tvs_setting_prop} and \ref{main_body_section}, we study the basic properties of functions with bounded real order
total variation, such as the lower semi-continuity (\emph{l.s.c.}) with respect to the function, and the  compactness of the embedding $BV^r(Q)\hookrightarrow L^1(Q)$.   The main result is (in which the space $SV^s$ is defined in Definition \ref{strict_conv_BVr}):
\begin{theorem}[see Theorem \ref{ATV_real_embedding}]\label{real_embedding_intro}
 Given a non-integer order $r\in \R^+\setminus \mathbb{N}$, let $s:=r-\lfloor r\rfloor \in (0,1)$
		be the
		decimal part of $r$. Consider a  sequence $\seqn{u_n}\subset SV^s(Q)\cap BV(Q)$ satisfying
	\be
	\sup\flp{\norm{u_n}_{L^\infty(\partial Q)}+\norm{u_n}_{BV^s(Q)}:\,\, n\in\N}<+\infty.
	\ee
	Then, there exists $u\in BV^s(Q)$ such that, up to a sub-sequence
	(which we do not relabel), $u_n\to u$ strongly in $L^1(Q)$.
\end{theorem}
Moreover, the recent research on neural architecture search (NAS) \cite{DBLP:conf/iclr/ZophL17, DBLP:journals/corr/abs-1806-09055} further speeds up the process of designing more powerful structures, and in certain cases, it involves to perform optimize regarding a set of loss functions, i.e., use model to determine the most suitable loss function from a set of functions. Hence, we will, in Section \ref{sec_functional_lsc}, investigate the functional properties of $\seqn{TV^{r_n}}$ with respect to a sequence of orders $\seqn{r_n}\subset \R^+$. The main theorem is:
\begin{theorem}[see Theorem \ref{compact_lsc_r_tv}]\label{compact_lsc_r_intro}
	Given sequences $\seqn{r_n}\subset \R^+$ and $\seqn{u_n}\subset L^1(Q)$ such that $r_n\to r\in\R^+\cup \flp{0}$, assume there exists $p\in(1,+\infty]$ such that 
	\be
	\sup\flp{\norm{u_n}_{L^p(Q)}+TV^{r_n}(u_n):\,\,n\in\N}<+\infty.
	\ee
	Then, the following statements hold.
	\begin{enumerate}[1.]
		\item
		There exists $u\in BV^r(Q)$, such that, up to a sub-sequence, $u_n\wto u$ weakly in $L^p(Q)$ and 
		\be
		\liminf_{n\to\infty} TV^{r_n}(u_n)\geq TV^r(u).
		\ee
		\item
		Assuming in addition that ${u_n}\in SV^{r_n}(Q)\cap BV(Q)$, $\norm{u_n}_{L^\infty(\partial Q)}$ is uniformly bounded, and $r_n\to r>0$. Then we have 
		\be
		u_n\to u\text{ strongly in }L^1(Q).
		\ee
	\end{enumerate}
\end{theorem}
Theorems \ref{real_embedding_intro} and \ref{compact_lsc_r_intro}
	provide the theoretical foundations for $TV^r$. These results
	provide a rigorous justification that not only $TV^r$ can be a loss function for any fixed order $r$, but also can be optimized regarding to its order in NAS, as more sophisticated approach in modern architecture design.\\

The paper is organized as follows: in Section \ref{tvs_setting_prop} we collect some notations and preliminary results on 
the fractional order derivative. In Section \ref{main_body_section} we analyze the main properties of the 
fractional $r$-order total variation, with 
{  fixed} $r\in\R^+\setminus \N$. The compact embedding, lower semi-continuity with respect to the order $r$, 
and {  the} relation between the fractional order total variation and its integer order counterpart will be 
the subjects of Section \ref{sec_functional_lsc}. 
\section{Preliminary results on fractional order derivatives}\label{tvs_setting_prop}
Through this article, $r\in \R^+$ denotes a (positive) constant, 
{  whose integer and fractional parts are denoted by $\ir$ and $s\in[0,1)$,
	respectively.} \\\\
We recall the definitions of fractional order derivative in dimension one.
\begin{define}[the fractional order derivative on unit interval]\label{frac_der_def}
	Let $I:=(0,1)$ and $x\in I$ be given.
	\begin{enumerate}[1.]
		
		\item
		The \emph{left (resp. right)-sided Riemann-Liouville} derivatives of order $r=\ir+s\in\R^+$ (see  \cite{MR1347689}) are defined by 
		\begin{align}
			\label{R_L_frac_1d_left}
			d^{r}_{L}w(x) &= \frac{1}{\Gamma(1-s)}\fsp{\frac d{dx}}^{\ir+1}\int_0^x \frac{w(t)}{(x-t)^{s}}dt,
			\qquad \Gamma(s):=\int_0^\infty e^{-t}t^{s-1}dt,
			\\
			d^{r}_{R}w(x) &= \frac{(-1)^{\ir+1}}{\Gamma(1-s)}\fsp{\frac d{dx}}^{\ir+1}\int_x^1 \frac{w(t)}{(t-x)^{s}}dt.
		\end{align}

		\item
		The \emph{left (resp. right)-Riemann-Liouville} fractional order integrals, of order $r\in\R^+$, are defined by 
		\be\label{r_int_frac_def}
		(\mathbb I_{L}^r w)(x):=\frac{1}{\Gamma(r)}\int_0^x \frac{w(t)}{(x-t)^{1-r}}dt
		\quad
		\text{ and }
		\quad
		({\mathbb I_R^r} w)(x):=\frac{1}{\Gamma(r)}\int_x^1 \frac{w(t)}{(x-t)^{1-r}}dt,
		\ee 
		{  respectively.}
		
		\item
		The {\em left (resp. right)-sided \emph{Caputo} derivative}, of order $r$,
		are defined by 
		\begin{align*}
			d^r_{L,c} w(x) &:=\frac{1}{\Gamma(1-s)}\int_0^x \frac{(d^{\ir+1})w(t)}{(x-t)^{s}}dt,\\
			d^{r}_{R,c}w(x) &:= \frac{(-1)^{\ir+1}}{\Gamma(1-s)}\int_x^1 \frac{(d^{\ir+1})w(t)}{(t-x)^{s}}dt.
		\end{align*}
	\end{enumerate}
\end{define}
We collect some immediate results regarding Definition \ref{frac_der_def} from
the literature.

\begin{remark} 
	Let $w\in C^\infty(\bar I)$ and $\phi\in C_c^\infty(I)$ be given.
	\begin{enumerate}[1.]
		\item
		The following integration by parts formulas hold (see, e.g.,\cite{1751-8121-40-24-003}):
		\begin{align*}
			\int_I w\,d_{R,c}^r \phi\,dx &=(-1)^{\ir+1}\int_I (d_L^rw)\, \phi\,dx,\\
			\int_I w\,d_{L,c}^r \phi\,dx&=(-1)^{\ir+1}\int_I (d_R^rw)\, { \phi}\,dx.
		\end{align*}
		
		\item
		The R-L and Caputo derivatives are equivalent for compactly supported functions (\cite[Theorem 2.2]{MR1347689}), i.e., for every $x\in I$, {  it} holds
		\be\label{caputo_eq_R-L}
		d^{r}_L\phi(x)= d^{r}_{L,c}\phi(x)\text{ and }d^{r}_R\phi(x)= d^{r}_{R,c}\phi(x).
		\ee
		Thus, the integration by parts formulas above can be rewritten as
		\begin{align*}
			\int_I w\,d_{R}^r \phi\,dx &=(-1)^{\ir+1}\int_I (d_L^rw)\, \phi\,dx ,\\
			\int_I w\,d_{L}^r \phi\,dx &=(-1)^{\ir+1}\int_I (d_R^rw)\, { \phi}\,dx.
		\end{align*}
		
		\item
		For {  given} functions $w_1$, $w_2\in C^\infty(\bar I)$, and
		{  parameters} $a$, $b\in\R$, {  the following linearity
			property holds:}
		\be\label{linearity_frac}
		d^r(a w_1(x)+b w_2(x))=a d^r w_1(x)+b d^r w_2(x).
		\ee
	\end{enumerate}

\end{remark}
\begin{remark}\label{left_use_whole}
	In the following, we shall only work with left-sided fractional order derivative/integrations, as the   
	arguments for the right-sided analogues
	are identical. 
	Thus, for brevity, we drop the underlying $L$ and $R$, 
	and write $d^r$ and $\mathbb I^r$, instead of $d^r_L$ and $\mathbb I^r_L$, unless 
	otherwise specified.
\end{remark}
We next recall the definition of representable functions.
\begin{define}[Representable functions]\label{frac_represent_I}
	We denote by $\mathbb I^r(L^1(I))$, $r>0$, the space of functions represented by the $r$-order derivative of a summable function. That is,
	\be
	\mathbb I^r(L^1(I)):=\flp{f\in L^1(I):\,\, f=\mathbb I^rw,\,\,w\in L^1(I)}.
	\ee
\end{define}
Next we recall several theorems, {  from \cite{MR1347689}}, 
on representable functions in dimension one.

\begin{theorem}\label{thm_MR1347689}
	For convenience, we unify our notation by writing $\mathbb I^r = d^{-r}$ for $r<0$.
	\begin{enumerate}[1.]
		
		\item\label{cite_represent_frac}
		{\cite[Theorem 2.3]{MR1347689}}
		{  In order for}
		$w(x)\in \mathbb I^r(L^1(I))$, $r>0$ 
		{  to hold, it is necessary and sufficient to have}
		\be
		\mathbb I^{\ir+1-r}[w](x)\in W^{\ir+1,1}(I),\qquad r=\ir+s \label{inte_frac_cond}
		\ee
		and
		\be
		(d^l \mathbb I^{\ir+1-r}[w])(0)=0, \qquad l=0,1,\cdots,\ir\label{bdy_frac_cond}
		\ee
		
		\item\label{MR1347689T2_5}
		{\cite[Theorem 2.5]{MR1347689}}
		The relation
		\be
		\mathbb I^{r_1} \mathbb I^{r_2} w=\mathbb I^{{r_1}+r_2}w
		\ee
		is valid if one of the following conditions holds:
		\begin{enumerate}[1.]
			
			\item
			$r_2>0$, ${r_1}+r_2>0$, provided that $w\in L^1(I)$,
			
			\item
			$r_2<0$, ${r_1}>0$, provided that $w\in \mathbb I^{-r_2}(L^1(I))$,
			
			\item
			${r_1}<0$, ${r_1}+r_2<0$, provided that $w\in \mathbb I^{-{r_1}-r_2}(L^1(I))$.
		\end{enumerate}
		
		\item\label{semigroup_frac_int}
		{\cite[Theorem 2.6]{MR1347689}}
		Let $r\in\R^+$ be given.
		\begin{enumerate}[1.]
			
			\item
			The fractional order integration operator $\mathbb I^r$ forms a semigroup in $L^p(I)$, $p\geq 1$, which is continuous in the uniform topology for all $r> 0$,
			and strongly continuous for all $r\geq 0$. 
			\item
			(\cite[(2.72)]{MR1347689}) It holds 
			\be\label{eq_semigroup_frac_int}
			\norm{\mathbb I^r w}_{L^{ p}(a,b)}\leq (b-a)^r\frac{1}{r\Gamma(r)}\norm{w}_{L^{ p}(a,b)},\qquad{ p
				\ge 1}.
			\ee
		\end{enumerate}
	\end{enumerate}
\end{theorem}
{ Next, we introduce }the notations for (partial) fractional order derivative in 
higher
dimension. %
\begin{define}[fractional order partial derivative]
	Given $x=(x_1,\ldots,x_N)\in Q=(0,1)^N\subset\rn$ and $u\in C^\infty(Q)$, we define the $r$-order partial derivative: 
	\begin{align*}
		\partial^r_1 u(x) := \frac {d^r}{ {  dt^r}} u(t,x_2,x_3,\ldots,x_N),
	\end{align*}
	and similarly for $\partial_i^ru(x)$, $i=2,\ldots, N$.
\end{define}
We next recall the integration by parts formula  {  for} fractional order
{  derivatives}
from \cite{chen2013fractional}: for $u\in C^\infty(Q)$, $v\in C_c^\infty(Q)$, it holds
\be
\int_Q u\,\partial^s v \,dx = - \int_Q \partial^s u\, v\,dx.
\ee
Then, by Theorem \ref{thm_MR1347689}, we have the following multi-index integration by parts formula:
\be
\int_Q u\,\partial^r v \,dx = (-1)^{\ir+1} \int_Q \partial^r u\, v\,dx.
\ee
We conclude this section by recalling the following technical lemma. 
\begin{lemma}[\cite{MR1347689}]\label{power_function_s}
	Let $k\in\N$ be given.
	The following assertions hold.
	\begin{enumerate}[1.]
		\item
		We have
		\be
		d^s_L x^k = \frac{\Gamma(k+1)}{\Gamma\fsp{k-s+1}}x^{k-s}
		,
		\quad
		{  d^s_R (1-x)^k = \frac{\Gamma(k+1)}{\Gamma\fsp{k-s+1}}(1-x)^{k-s}
			,\qquad \text{for all } x\in I,
			s\in(0,1) }.
		\ee
		\item
		If $k=0$, i.e., $w(x):=x^k=1$ {  is the function
			identically equal to 1, then $d^r_L w(x)=d^r_R w(x)=0$} if and only if $r\in\N$.
		
		\item
		For all $s\in(0,1)$, we have {  $d^s_L x^{s-1}=d^s_R (1-x)^{s-1}=0$}.
	\end{enumerate}
\end{lemma}
\section{The space of functions with bounded fractional-order total variation}\label{main_body_section}
\subsection{Total variation with different underlying Euclidean norm}
We start by recalling the definition of Euclidean $\ell^p$-norm on $\rn$. Let $p\in[1,+\infty)$ and $x=(x_1,x_2,\ldots,x_N)\in \rn $ be given, we define 
\be
\abs{x}_{\ell^p}:=(|x_1|^p+|x_2|^p+\cdots |x_N|^p)^{1/p}.
\label{norm inequality}
\ee
Note that $\abs{\cdot}_{\ell^p}$ are equivalent norms on $\rn$, in the sense that,
{  for all $1\leq q<p\leq \infty$,}
\be\label{x_eu_equivalence}
\abs{x}_{\ell^p}\leq\abs{x}_{\ell^q}\leq N^{1/q-1/p}\abs{x}_{\ell^p}.
\ee

\begin{define}\label{TV_ATV_integer}
	Let $u\in L^1(Q)$ be given. We recall the following definition:
	  given $k\in\N$,
		the $k$-th order total variation, with underlying Euclidean $\ell^p$-norm, is defined as
		\be
		TV_{\ell^p}^k(u):=\sup\flp{\int_Qu\, \divg^k \vp \,dx:\,\,\vp\in C_c^\infty(Q;\M^{N\times N^{k-1}})\text{ and }\abs{\vp}_{\ell^p}^\ast\leq 1}.
		\ee
		Here $\abs{\cdot}_{\ell^p}^\ast$ denotes the dual norm associated with $\abs{\cdot}_{\ell^p}$,
		
			in the sense that
			\[\abs{\varphi}_{\ell^p}^* = \sup_{  i} \sup_x |\varphi_{  i} (x)|_{\ell^{p^*}},
			\qquad
			\frac{1}{p}+\frac{1}{p^*}=1,\] 
			where $\varphi_{  i}$, {  $i=1,\cdots ,N^{k}$,}
			are the individual components of $\varphi$.\\
		
		Of particular interest are the
		first and second order total variation, corresponding to the
		cases $k=1,2$:
		\begin{align*}
			TV_{\ell^p}(u)&:=\sup\flp{\int_Qu\, \divg \vp \,dx:\,\,\vp\in C_c^\infty(Q;\rn)\text{ and }\abs{\vp}_{\ell^p}^\ast\leq 1},\\
			TV_{\ell^p}^2(u)&:=\sup\flp{\int_Qu\, \divg^2 \vp \,dx:\,\,\vp\in C_c^\infty(Q;\M^{N\times N})\text{ and }\abs{\vp}_{\ell^p}^\ast\leq 1}.
		\end{align*}
\end{define}

As an example, when $N=2$, we have $\vp=[\vp_1,\vp_2;\vp_3,\vp_4]$, and 
\begin{align*}
	\divg^2 \vp = \divg(\divg(\vp_1,\vp_2),\divg(\vp_3,\vp_4))&=\divg(\partial_1 \vp_1+\partial_2 \vp_2, \partial_1\vp_3+\partial_2\vp_4)\\
	&=\partial_1\partial_1\vp_1+\partial_1\partial_2\vp_2+\partial_2\partial_1 \vp_3+\partial_2\partial_2\vp_4.
\end{align*}
{  We remark that it is possible to recover} the classical \emph{isotropic total variation} ($TV$) and \emph{an-isotropic total variation} ($ATV$) by letting $p=2$ and $p=1$, respectively. Moreover, in dimension one, all $TV_{\ell^p}$ are the same.

More in general, the following equivalence result holds.
\begin{lemma}\label{TVlpq_lemma}[Equivalence of the $TV_{\ell^p}^r$
	norms]
	Given $1\leq q<p\leq \infty$, we have
	\be
	N^{1/p-1/q}TV_{\ell^q}^r(u)\leq TV_{\ell^p}^r(u)\leq TV_{\ell^q}^r(u),
	\ee
	for all $r\in\R^+$ and $u\in BV^r(Q)$.
\end{lemma}

\begin{proof}
	{ 	In the following, we will use $p^*$ to denote the dual exponent of $p$, i.e.
		$\frac{1}{p}+\frac{1}{p^*}=1$.
	}
	Let $1\leq q<p\leq+\infty$ be given. From Remark \ref{an_iso_equ}, we have 
	\be\label{lp_euclidean_use}
	N^{1/p-1/q}\abs{\vp(x)}_{\ell^{p^\ast}}\leq \abs{\vp(x)}_{\ell^{q^\ast}}\leq \abs{\vp(x)}_{\ell^{p^\ast}},
	\ee
	for all $\vp\in C_c^\infty(Q;\rn)$, and $x\in Q$. That is, for any $\vp\in C_c^\infty(Q;\rn)$ such that $\abs{\vp(x)}_{\ell^{p^\ast}}\leq 1$, 
	we have $\abs{\vp(x)}_{\ell^{q^\ast}}\leq 1$. 
	{ 
		Thus
		\[\left\{  \vp\in C_c^\infty(Q;\M^{N\times (N^{\ir})}):\abs{\vp}_{\ell^q}^\ast\leq 1  \right\}
		\supseteq \left\{  \vp\in C_c^\infty(Q;\M^{N\times (N^{\ir})}):\abs{\vp}_{\ell^p}^\ast\leq 1  \right\},\]
		which then gives
		$
		TV^r_{\ell^q}(u)\ge 	 TV^r_{\ell^p}(u)$. 
		The proof for $N^{1/p-1/q}TV_{\ell^q}^r(u)\leq TV_{\ell^p}^r(u)$ is completely analogous.
	}

\end{proof}

We recall the usual trace operator for function with bounded total variation.

\begin{theorem}[{\cite[Theorem 2, Page 181]{evans2015measure}}]\label{usual_trace}
	Let $u\in BV(Q)$ be given. Then, for $\hnmo$-a.e. $x_0\in\partial Q$,
	\be
	\lime\fint_{B(x_0,\e)\cap Q}\abs{u- T[u](x)}dx=0,
	\ee
	where $T[\cdot]$ denotes the standard trace operator. 
	{  That is,}
	\be
	T[u](x_0)=\lime \fint_{B(x_0,\e)\cap Q}u(x)\,dx.
	\ee
\end{theorem}

\begin{define}\label{isotropic_frac_variation}
	We define the \emph{$r$-order total variation} $TV_{\ell^p}^r(u)$ on $u\in L^1(Q)$ as follows.
	\begin{enumerate}[1.]
		\item
		For $r=s\in(0,1)$ (i.e. $\ir=0$), we define
		\be\label{RLFOD}
		TV_{\ell^p}^s(u):=\sup\flp{\int_Qu\, \divg^s \vp \,dx:\,\,\vp\in C_c^\infty(Q;\rn)\text{ and }\abs{\vp}_{\ell^p}^\ast\leq 1},
		\ee  
		where 
		\be\label{scaled_divg}
		\divg^s u:=[(1-1/N)s+1/N]\sum_{i=1}^N\partial^s_{i,R}\vp_i; 
		\ee
		\item
		For $r=\ir+s$ where $\ir\geq1$, we define 
		\be\label{RLFODr}
		TV_{\ell^p}^r(u):=\sup\flp{\int_Qu\, \divg^s[\divg^\ir \vp] \,dx:\,\,\vp\in C_c^\infty(Q;\M^{N\times (N^{k})})\text{ and }\abs{\vp}_{\ell^p}^\ast\leq 1}.
		\ee
	\end{enumerate}
\end{define}
\begin{remark}
	In \eqref{scaled_divg} we applied the right-sided derivative on the test function $\vp$.
	{  In this way,} when $u$ is sufficient regular, {  e.g.} $u\in C^\infty(\bar Q)$, the integration by parts formula 
	\be
	\int_Qu\, \divg^s \vp \,dx =- \int_Q\nabla^s_Lu\,\vp\,dx
	\ee 
	{  holds.}
	{  Similarly,}
	if we choose to work primarily with the right-sided derivative, we shall use left-sided derivative on 
	the test function $\vp$ in \eqref{scaled_divg}.
\end{remark}
%
%
\begin{define}\label{strict_conv_BVr}
	Let $r\in\R^+$ and $p\in[1,+\infty]$ be given. 
	\begin{enumerate}[1.]
		\item
		Given a sequence $\seqn{u_n}\subseteq L^1(Q)$ such that $TV^r(u_n)<+\infty$ for each $n\in\N$, we say it is \emph{strictly} converging to $u\in L^1(Q)$ 
		with respect to the $TV^r_{\ell^p}$ semi-norm, and write $u_n\wtogsp u$, if
		\be\label{def_eq_strict_covg}
		\limn\norm{u-u_n}_{L^1(Q)}+\abs{TV^r_{\ell^p}(u_n)-TV_{\ell^p}^r(u)}= 0.
		\ee
		That is, $u_n\wtogsp u$ if $u_n\to u$ strongly in $L^1(Q)$ and $TV^r_{\ell^p}(u_n)\to TV_{\ell^p}^r(u)$.

		\item
		We define the space $SV^r(Q)$ by
		\be SV^r(Q):=\bigcap_{p\in[1,+\infty]} \overline{C^\infty(Q)}^{\text{  s-s }(p)} .
		\label{BV_r_p_diff_ss} \ee
		\item
		We define the (standard) $BV^r(Q)$ space by 
		\be\label{BV_r_p_diff}
		BV^r(Q):=\bigcap_{p\in[1,+\infty]}\flp{u\in L^1(Q):\,\, TV^r_{\ell^p}(u)<+\infty}.
		\ee
	\end{enumerate}
\end{define}
\begin{remark}[Equivalence between $TV_{\ell^p}^r$]\label{an_iso_equ}
	Definition \ref{strict_conv_BVr} has several consequences.
	\begin{enumerate}[1.]
		\item
		By \eqref{x_eu_equivalence} we have, for any $1\leq q<p\leq+\infty$,
		\be\label{equ_p_1_p}
		N^{1/p-1/q}TV_{\ell^p}^r(u)\leq TV_{\ell^q}^r(u)\leq TV_{\ell^p}^r(u).
		\ee
		That is, the set $\flp{u\in L^1(Q):\,\,TV^r_{\ell^p}(u)<+\infty}$ is 
		{  actually}
		independent of $p$. As a consequence,
		{  the functional space $BV^r(Q)$, defined in \eqref{BV_r_p_diff},
			satisfies}
		\be
		BV^r(Q)=\flp{u\in L^1(Q):\,\,TV^r_{\ell^2}(u)<+\infty},
		\ee
		without {  any} dependence on the underlying $\ell^p$-norm.
		\item 
		The space $SV^r(Q)$, defined in \eqref{BV_r_p_diff_ss}, enjoys the ``smooth approximation" property: for each $u\in SV^r(Q)$, there exists 
		a sequence $\seqn{u_n}\subseteq C^\infty(Q)\cap BV^r(Q)$ such that \eqref{def_eq_strict_covg} holds. 
		In the case of integer order, i.e. $r=k\in\N$, we do have $SV^k(Q)=BV^k(Q)$ (see for instance \cite{evans2015measure}). 
		However, due to the singularities at the boundary arising from the definition of fractional derivatives, we are unable to prove a smooth approximation result in this case.
		In particular, the construction from \cite{evans2015measure} would not work, unless additional conditions are imposed. 	\end{enumerate}
\end{remark}

\BLK
We conclude this subsection with some definitions and properties of fractional order Sobolev semi-norms.
\begin{define}
	{ 
		Given $q\in[1,+\infty]$, $r>0$}, we define the $r$-order fractional Sobolev space  
	\begin{align*}
		W^{r, {  q}}(Q):=\bigg\{u\in L^ {  q}(Q 
		):\,\, 
		&\text{there	exists } g\in L^1(Q{  ;{\mathbb{M}}^{N\times N^{\lfloor r\rfloor}} })
		\text{ such that }\int_Q u\, \divg_R^r \vp \,dx=\int_Q g
		{ \cdot}\vp\,dx\\
		&\text{ for all test functions } \varphi\in C_c^\infty(Q {  ; {\mathbb{M}}^{N\times N^{\lfloor r\rfloor}} } ) \bigg\}.
	\end{align*}
	Here $g\in L^{  q}(Q {  ;{\mathbb{M}}^{N\times N^{\lfloor r\rfloor}} } 
	)$ is the weak $r$-order fractional derivative of $u$. We equip it with norm
	\be
	\norm{u}_{W^{r,{ q}}_{\ell^p}(Q)}:=\norm{u}_{L^{ q}(Q)}
	+\int_Q\abs{g}_{\ell^p}dx.
	\ee
\end{define} 

\begin{define}
	Let $r=\ir+s$. By $AC^{r,1}(I)$ we denote the set of all functions $w:I\to\R$ admitting a representation
	of the form
	\be\label{repre_AC}
	w(t)=\sum_{i=0}^\ir \frac{c_i}{\Gamma(s+i)}t^{s-1+i}+\mathbb I^r \phi(t),\,\, t\in I\,\,a.e.,
	\ee
	where $c_0,\ldots,c_k\in\R$ and $\phi\in L^1(I)$.
\end{define}

We recall the following results form \cite{MR3144452}.
\begin{theorem}\label{thm_MR3144452}
	Let $r=\ir+s$ be given.
	\begin{enumerate}[1.]
		\item
		\cite[Theorem 7]{MR3144452} { A} function $w\in L^1(I)$ admits the $r$-order derivative if and only if $w\in AC^{r,1}(I)$. In this case, $w$ has the representation \eqref{repre_AC}, and
		\begin{align}\label{d_representable}
			d^{i+s} w(0)=c_i,\,\, i=0,\cdots, k-2,\text{ and }d^r w(t)=\phi(t),\,\, t\in I\,\,a.e..
		\end{align}
		\item
		\cite[Theorem 19]{MR3144452} We have
		\be\label{ac_sobolev_s}
		W^{r,1}(I)=AC^{r,1}(I)\cap L^1(I).
		\ee
	\end{enumerate}
\end{theorem}

We remark that, from \eqref{d_representable} and \eqref{ac_sobolev_s}, 
for a given $w\in W^{r,1}(I)$, the corresponding $\phi{  = d^rw}\in L^1(I)$ satisfies $\norm{\mathbb I^r\phi}_{L^1(I)}<+\infty$, 
and hence ${  w}\in \mathbb I^r(L^1(I))$ in view of Definition \ref{frac_represent_I}.
\begin{remark}\label{int_iso_equ}
	Let $p\in[1,+\infty]$ be given, and assume that $u\in C^\infty(Q)\cap BV^{s}(Q)$. Then, 
	{  by \cite[Proposition 3.5]{zhang2015total},}
	we have
	\be\label{TV_sobolevl_equiv}
	TV^s_{\ell^p}(u) =\int_Q \abs{\nabla^s u}_{\ell^p}dx=\int_Q \abs{(\partial_1^s u,\cdots,\partial^s_{ N} u)}_{\ell^p}dx.
	\ee
	\\\\

	Equation \eqref{TV_sobolevl_equiv} allows us to use the fact that the an-isotropic total variation $TV_{\ell^1}$ can be computed axis by axis. For example, 
	{ for} $N=2$, by \eqref{TV_sobolevl_equiv} we have
	\be
	TV^s_{\ell^1}(u) =\int_Q \abs{(\partial_1^s u,\partial_2^su)}_{\ell^1}dx= \int_Q \abs{\partial^s_1 u}dx+\int_Q \abs{\partial^s_2 u}dx.
	\ee
	That is, we are able to separate the integration in the two variables. This will be crucial in allowing us to study
	properties of real order total variation in the multi-dimensional setting, by using results from the (often easier) { one-dimensional} case. 
	
\end{remark}

\subsection{Basic properties of $TV^r$ with fixed order of derivative}\label{sec_smooth_approx}
For brevity, we will only write $TV^r$ without explicit reference to the underlying Euclidean $\ell^p$-norm. 
However, in several arguments it will be advantageous to use the $TV_{\ell^1}^r$ semi-norm (see Remark \ref{int_iso_equ}). 
In such instances, we will thus write $TV_{\ell^1}^r$ to clarify that we are relying on the underlying norm being the Euclidean $\ell^1$-norm.
We first prove a lower semi-continuity result with fixed order $r\in\R^+$.

\begin{theorem}  \label{weak_star_comp_s} 
	Given $u\in \mb(Q)$, { i.e. the space of finite Radon measures
		on $Q$,} and sequence $\seqn{u_n}\subseteq BV^r(Q)$ satisfying one of the following conditions:
	\begin{enumerate}[1.]
		\item
		$u\in L^1(Q)$ and $\seqn{u_n}$ is locally uniformly integrable and $u_n\to u$ a.e.,
		\item 
		$u_n\wtos u$ in $\mb(Q)$
		\item 
		$u_n\wto u$ in $L^p(Q)$, for some $p>1$,
	\end{enumerate}
	then we have
	\be\label{liminf_s_weak}
	\liminf_{n\to\infty} TV^r(u_n)\geq TV^r(u).
	\ee
\end{theorem}
Note that (3) is stronger than (2), but we stated it explicitly since it is a special case widely used in this paper.
To prove Theorem \ref{weak_star_comp_s}, a preliminary result is required. 

\begin{lemma}\label{linftyds}
	Let $\vp\in C_c^\infty(Q)$ be given. Then the following statements hold.
	\begin{enumerate}[1.]
		\item 
		For any fixed $T\in \N$, we have
		\be
		\sup\flp{\norm{d^r\vp}_{L^\infty(Q)}:\,\,{r\in(0,T)} }<+\infty.
		\ee
		\item
		For a.e. $x\in Q$, we have
		\be
		\divg^r\vp(x)\to \divg^\ir\vp(x)\text{ as }r\to\ir^+
		\ee
		and 
		\be
		\divg^r\vp(x)\to \divg^{\lceil r \rceil}\vp(x)\text{ as }r\to \lceil r \rceil^-.
		\ee
	\end{enumerate}
\end{lemma}

\begin{proof}

	We first consider the { one-dimensional} case, i.e. $Q=I=(0,1)$. 
	In view of \eqref{caputo_eq_R-L}, for any $r=\ir+s$, we have
	\begin{align*}
		\abs{d^r_L \varphi(x)}& = \abs{d^r_{L,c}(x)}\leq\frac{1}{\Gamma(1-s)}\int_0^x \frac{\abs{(d^{\ir+1}){ \varphi}(t)}}{(x-t)^{s}}dt\\
		&\leq \frac{\norm{{ \varphi}}_{W^{\ir+1,\infty}(I) } }{\Gamma(1-s)}\int_0^x \frac{1}{(x-t)^{s}}dt
		\leq \frac{\norm{{ \varphi}}_{W^{\ir+1,\infty}(I) }}{(s-1)\Gamma(1-s)}\fmp{(x-t)^{1-s}\bigg|^x_0}
		\leq 
		\frac{\norm{{ \varphi}}_{W^{\ir+1,\infty}(I) }}{(1-s)\Gamma(1-s)},
	\end{align*}
	which implies 
	\be
	\abs{d^r_L { \varphi}(x)} \leq\norm{d^{\ir+1}{ \varphi}}_{L^{\infty}(I)}\cdot\frac{1}{(1-s)\Gamma(1-s)}.
	\ee
	Note also that, by Euler's reflection formula,
	\be
	\sup_{s\in(0, 1)}\frac{1}{\Gamma(1-s)(1-s)}=\sup_{s\in(0,1)}\frac{1}{\Gamma(2-s)}\leq1,
	\ee
	hence, 
	\be
	\norm{d_L^r { \varphi}}_{L^\infty(I)}\leq \norm{d^{\ir+1}
		{ \varphi}}_{L^{\infty}(I)} <+\infty,
	\ee
	and
	\be
	\sup\flp{\norm{d^r_L\vp}_{L^\infty(I)}:\,\,{r\in(0,T)} }\leq \sum_{l=0}^{T}\norm{d^{l+1}\vp}_{L^{\infty}(I)}<+\infty.
	\ee
	The same arguments give the desired results for the right and central sided $R-L$ derivatives.\\\\
	We next prove Statement 2, for the case $0<r<1$. 
	The case $r\geq 1$ is proven using similar arguments. In view of \eqref{caputo_eq_R-L}, we have $d^s\vp=d^s_c \vp$, with
	$d^s_c $ denoting the \emph{Caputo} fractional derivative. Recall the Laplace transform gives
	\be
	\mathcal L\flp{d^s_c \vp}(y)=y^s\mathcal L\flp{\vp}(y)-y^{s-1}\vp(0)=y^s\mathcal L\flp{\vp}(y).
	\ee
	Therefore, we have 
	\be
	\lim_{s\nearrow 1}\mathcal L\flp{d^s_c \vp}(y) = y\,\mathcal L\flp{\vp}(y) 
	\qquad\text{ and }\qquad\lim_{s\searrow 0}\mathcal L\flp{d^s_c \vp}(y) = \mathcal L\flp{\vp}(y),
	\ee
	and hence we conclude  
	\be
	\lim_{s\nearrow 1}d^s_c(\vp)(x)= d\vp(x)\qquad\text{ and }\qquad\lim_{s\searrow 0}d^s_c(\vp)(x)= \vp(x),
	\ee
	as desired. \\\\
	The multi-dimensional case can be directly inferred from { the one-dimensional} case.
	We discuss the two-dimensional case (i.e. $N=2$) as an example. 
	{ 
		To show the first statement, i.e.
		\[ \sup\left\{ \|d^r \varphi\|_{L^\infty(Q)}: r\in(0,T) \right\}<+\infty,
		\qquad \text{for any } \varphi\in C_c^\infty(Q),
		\]
		it suffices to note that
		\begin{align*}
			| \partial_1^r \varphi(x_1,x_2) | & \le \frac{1}{\Gamma(1-s)} \int_0^{x_1} \frac{|\partial_1^{ \ir+1 } \varphi(t,x_2)  |}{(x_1-t
				)^s} d t 
			\qquad (r=\ir+s)\\
			&\le
			\|\nabla \varphi \|_{L^\infty(Q)} \frac{1}{\Gamma(1-s)} \int_0^{x_1} \frac{1}{(x_1-t)^s} d t\le \|\nabla \varphi \|_{L^\infty(Q)} \frac{1}{(1-s)\Gamma(1-s)}.
		\end{align*} 
		To show the second statement, recall that 
		\[\divg^s \vp(x_1,x_2)=\partial^s_1\vp(x_1,x_2)+\partial_2^s\vp(x_1,x_2),\] 
		thus we have $\partial^s_1\vp(x_1,x_2) = d^s\vp(\cdot,x_2)$ for fixed $x_2\in (0,1)$. 
		Then, by using the conclusion in the 1D case,
		we infer that, for $s\to 0^+$,
		\[ \partial_1^s \vp(x_1,x_2) \to \vp(x_1,x_2),\qquad \text{for a.e. $x_1$ and fixed $x_2$}, \]
		and similarly,
		\[ \partial_2^s \vp(x_1,x_2) \to  \vp(x_1,x_2),\qquad \text{for a.e. $x_2$ and fixed $x_1$}. \]
		The proof is thus complete.
	}
\end{proof}

\begin{proof}[Proof of Theorem \ref{weak_star_comp_s}]
	Let $r=\ir+s$ be given, and fix an arbitrary 
	$\vp\in C_c^\infty(Q;\M^{N\times N^{ { \ir} }})$, 
	{  $|\vp|_{\ell^2}^*\le 1$}. In view of \eqref{RLFOD} (or \eqref{RLFODr}) we have
	\be
	TV^r(u_n)\geq \int_Qu_n\, [\divg^s\divg^\ir \vp] \,dx.
	\ee
	{ 
		By Lemma \ref{linftyds}, 
		from
		$\varphi\in C_c^\infty(Q;\M^{N\times N^{ { \ir} }})  $, we get
		$\divg^\ir \varphi$ is also smooth and compactly { supported}, which 
		in turn gives $\divg^s\divg^\ir \vp\in L^\infty(Q)  $.
	}
	
	\medskip
	
	{ 
		Assume} condition (1) holds. Since $\seqn{u_n}$ is locally uniformly integrable and $u_n\to u$ a.e., 
	by the dominated convergence theorem 
	{ 
		\begin{align}
			\limn\, \int_Qu_n\, [\divg^s\divg^\ir \vp] \,dx =\int_Qu\, [\divg^s\divg^\ir \vp] \,dx.\label{new numbered eq 1}
		\end{align}
	}
	If we assume either condition (2), or the stronger (3), then note that
	since $\vp\in C_c^\infty(Q;\mathbb M^{N\times N^\ir})$, we have $\divg^s\divg^\ir\vp\in C(\bar Q)$. Hence, since $\seqn{u_n}\subset L^1(Q)$, in view of the $\text{weak}^\ast$ convergence in $\mb(Q)$ (see \cite[Page 116]{brezis2010functional}), we have again { \eqref{new numbered eq 1}}.

	Thus, in all cases, we have
	\be
	\liminfn\, TV_{\ell^1}^r(u_n)\geq \liminfn\, \int_Qu_n\, [\divg^s\divg^\ir \vp] \,dx =\int_Qu\, [\divg^s\divg^\ir \vp] \,dx.
	\ee
	Taking the supremum over all $\vp\in C_c^\infty(Q;\M^{N\times N^\ir})$ with {  $|\vp|_{\ell^2}^*\le 1$}, we conclude \eqref{liminf_s_weak}, as desired. 
\end{proof}

{ 
	Now we show that it is possible to approximate functions in $SV^r(Q)$ with smooth functions.
	
	\begin{proposition}[strict approximation with smooth functions]\label{approx_smooth}
		Let $r\in\R^+$ and $u\in SV^r(Q)$ be given. Then there exists a sequence $\seqn{u_n}\subset C^\infty(\bar Q)\cap BV^r(Q)$ such that 
		\be\label{strictly_ATV_approx}
		u_n\to u\text{ strongly in }L^1(Q)\text{ and }\limn {TV^r}(u_n)={TV^r}(u).
		\ee
	\end{proposition}
	
	To this aim, two preliminary results are required. First,
	we recall the \emph{Riesz} representation theorem. }

\begin{theorem}[{\cite[Theorem 1, Section 1.8]{evans2015measure}}]\label{riesz_theorem}
	Let $L:$ $C_c(\rn,\R^M)\to\R$ be a linear functional satisfying
	\be
	\sup\flp{L(\vp):\,\,\vp\in C_c(\rn;\R^M),\,\,{  |\vp|_{\ell^2}^*\le 1},\,\, {  \operatorname{support}} (\vp)\subset K}<+\infty
	\ee
	for some compact set $K\subset \rn$. Then there exists a Radon measure $\mu$ on $\rn$, and a $\mu$-measurable function $\sigma$: $\rn\to\R^M$ such that
	\begin{enumerate}[1.]
		\item
		$\abs{\sigma(x)}=1$ for $\mu$-a.e. $x$, and
		\item
		$L(\vp)=\int_\rn \vp\cdot\sigma d\mu$.
	\end{enumerate}
\end{theorem}

Next, { we show the following crucial result:} 
\begin{lemma}\label{jinleyidian}
	Given a function $u\in BV^r(Q)$, there exists a Radon measure 
	$\mu$ on $Q$ and a $\mu$-measurable function $\sigma: Q\to\rn$ such that 
	\begin{enumerate}[1.]
		\item
		$\abs{\sigma(x)}=1$ $\mu$-a.e., and
		\item
		$\int_Q u\,\divg^r \vp\,dx=-\int_Q \vp\cdot\sigma \,d\mu$ for all $\vp\in C_c^\infty(Q;\rn)$, {  $|\vp|_{\ell^2}^*\le 1$}.
	\end{enumerate}
\end{lemma}
\begin{proof}
	We first define the linear functional 
	\be
	L:\,\, C_c^\infty(Q;\mathbb M^{N\times N^\ir})\to\R,
	\qquad L(\vp):=-\int_Q u\,\divg^r\vp\,dx.
	\ee
	Since $TV^r(u)<+\infty$, we have, in view of \eqref{RLFODr},
	\be
	\sup\flp{\frac1{\norm{\vp}_{L^\infty(Q)}}\int_Qu\, \divg^r \vp \,dx:
		\vp\in C_c^\infty(Q;\M^{N\times N^\ir})}=TV^r(u)<+\infty.
	\ee
	Thus, 
	\be\label{smooth_define}
	\abs{L(\vp)}\leq TV^r(u)\norm{\vp}_{L^\infty(Q)}.
	\ee
	Now we extend by continuity the definition of $L$ to the entire space $C_c(Q;\rn)$:
	for an arbitrary $\vp\in C_c(Q;\rn)$, we consider the mollifications $\vp_\e:=\vp\ast\eta_\e$ 
	(for some {  uninfluential} mollifier $\eta_e$) and, by \cite[Theorem 1, item (ii), Section 4.2]{evans2015measure}, 
	\be\label{uniform_approach}
	\vp_\e\to\vp\text{ uniformly on }Q.
	\ee
	Therefore, by defining
	\be
	\bar L(\vp):=\lime L(\vp_\e)\text{ for }\vp\in C_c(Q;\rn),
	\ee
	in view of \eqref{smooth_define} and \eqref{uniform_approach}, we conclude that 
	\be
	\sup\flp{\bar L(\vp):\text{ for }\vp\in C_c(Q;\M^{N\times N^{k}})\text{ and }{  |\vp|_{\ell^2}^*\le 1}}<+\infty.
	\ee
	Thus, by Theorem \ref{riesz_theorem}, the proof is complete.
\end{proof}

\begin{proof}({ of Proposition \ref{approx_smooth}})
	Let $u\in SV^r(Q)$ be given. In view of Definition \ref{strict_conv_BVr}, there exists a sequence $\seqn{u_n}\subset C^\infty(Q)$ such that 
	\be\label{wtogs_un_u}
	u_n\wtogs u.
	\ee
	Next, let $x_0:=(1/2,\cdots,1/2)$ be the center of $Q$. For sufficiently small $\e>0$, we define
	\be\label{u_n_eps_construct}
	u_n^\e(x):=u_n\Big(\frac{x-x_0}{1+\e}+x_0 \Big)\text{ for }x\in Q.
	\ee
	Note that by construction, $u_n^\e$ will be a scaled version of the restriction of $u_{n}$ to $Q_\e:={Q}/\fsp{1+\e}$.
	Since $u_n\in C^\infty(\overline{Q_\e})$, we have $u_n^\e\in C^\infty(\bar Q)$ too. \\\\
	We next show that $TV^r(u^\e_n)\to TV^r(u_n)$. Let $\e>0$ be fixed, then {  by Lemma \ref{jinleyidian},}
	\be\label{smaller_set_tv}
	TV^r(u_n)\geq \sup\flp{\int_{Q_\e} u_n\divg^r\vp\,dx:\,\,\vp\in C_c^\infty(Q_\e), {  |\vp|_{\ell^2}^*\le 1}}.
	\ee
	On the other hand, for any $\vp\in C_c^\infty (Q)$, we have
	\begin{align*}
		\int_Q u_n^\e\divg^r\vp\,dx 
		&=\int_Q u_n\fsp{(x-x_0)/(1+\e)+x_0}\divg^r\vp(x)\,dx\\
		&\leq (1+\e)^{\ir+1} \int_{Q_\e} u_n(y)\,\divg^r\vp((y-x_0)(1+\e)+x_0) dy\\
		& \overset{\eqref{smaller_set_tv}}{\leq} (1+\e)^{\ir+1} TV^r(u_n).
	\end{align*}
	{  Therefore,}
	\be\label{limsup_simple_geometry}
	TV^r(u_n^\e)\leq (1+\e)^{\ir+1} TV^r(u_n).
	\ee
	On the other hand, in view of \eqref{u_n_eps_construct}, we have $u_n^\e\to u_n$ strongly in $L^1(Q)$. By {  Theorem }  \ref{weak_star_comp_s}, we obtain  
	\be
	\liminf_{\e\to 0}TV^r(u_n^\e)\geq TV^r(u_n).
	\ee
	Combined with \eqref{limsup_simple_geometry}, {  this implies} $u_n^\e\wtogs u_n$. 
	Combined with \eqref{wtogs_un_u}, we infer the existence of a sub-sequence 
	$\seqn{u_{\e_n}}\subset C^\infty(Q)$ such that $u_{\e_n}\wtogs u$, 
	{  hence \eqref{strictly_ATV_approx} is proven}.
\end{proof}

\begin{remark}\label{smooth_use_here}
	We again emphasize that, in view of Remark \ref{int_iso_equ} and \cite[Proposition 3.5]{zhang2015total}, {  for any} $u\in C^\infty(Q)\cap BV^r(Q)$,
	{  it holds}
	\be
	TV_{\ell^1}^{l+s}(u)=\int_Q \abs{\nabla^{l+s}u}dx=\sum_{\abs{\alpha_s}=l+1}\int_0^1\cdots\int_0^1\abs{\partial^{\alpha_s} u(x_1,\ldots,x_N)}dx_1\cdots dx_N.
	\ee
	{ This allows us to compute 
		$TV_{\ell^1}^{l+s}(u)$ 
		by computing 
		$TV^{l'+s}(u)$, $l'=0,1,\cdots, l$, along each coordinate axis.}
\end{remark}

\section{Analytic properties of the functional space $SV^r(Q)$}\label{sec_functional_lsc}

\begin{remark}
	As the goal of this article is to construct models for imaging applications, we assume the function $u\in L^1(Q)$ analyzed here represents an image.
	That is, we could restrict ourself to consider only functions in the space
	\be\label{image_function}
	\IM(Q):=\flp{u\in BV(Q):\,\,\norm{T[u]}_{L^\infty(\partial Q)}\leq 1}.
	\ee
	{  Here $T:Q\longrightarrow \pd Q$ denotes the trace operator.}
	We remark that most of the conclusions in this article are independent of \eqref{image_function}. However, certain results can be improved 
	with \eqref{image_function}. We will always make it explicit when we do assume \eqref{image_function}. We also note that assumption \eqref{image_function} is crucial in the numerical realization of fractional order derivatives. We refer the interested reader to \cite[Section 4]{chen2013fractional} and the references therein.
\end{remark}
\subsection{Compact embedding for $BV^s$ with fixed $s\in(0,1)$}\label{sec_compact_fixed_s}
We start by recalling the following theorem from \cite{brezis2010functional}.
\begin{theorem}[{\cite[Theorem 4.26]{brezis2010functional}}]\label{lp_compactness}
	Let $\mathcal F$ be a bounded set in $L^p(\rn)$ with $1\leq p<+\infty$. Assume that 
	\be
	\lim_{\abs{h}\to 0}\norm{\tau_h f-f}_{L^p(\rn)}=0, \qquad
	{  \tau_h f(\cdot):=f(\cdot+h)}
	\ee
	uniformly in $f\in\mathcal F$.
	Then, the closure of $\mathcal F\lfloor \Om$ in $L^p(\Om)$ is compact for any measurable set $\Om\subset \rn$ with finite measure.
\end{theorem}
The main result of Section \ref{sec_compact_fixed_s} reads as follows.
\begin{theorem}[compact embedding $BV^s(Q) \hookrightarrow L^1(Q)$]\label{ATV_real_embedding}
	Let $s\in (0,1)$ be given. Assume $\seqn{u_n}\subset \text{ss-}BV^s(Q)$ satisfy
	\be\label{image_condition_infty}
	\sup\flp{\norm{u_n}_{L^\infty(\partial Q)}+\norm{u_n}_{BV^s(Q)}:\,\, n\in\N}<+\infty.
	\ee
	Then, there exists $u\in BV^s(Q)$ such that, upon sub-sequence, $u_n\to u$ strongly in $L^1(Q)$.
\end{theorem}
We prove Theorem \ref{ATV_real_embedding} in several steps. We first recall a revisited left sided $R-L$ $s$-order derivative
from \cite{MR2237634},
which we denote by $\hat d_L^s w(x)$, as follows:
\be\label{revisied_RL}
\hat d^s_Lw(x):=\frac1{\Gamma(1-s)} \frac d{dx}\int_0^x \frac{w(t)-w(0)}{\fsp{x-t}^s}dt.
\ee
That is, the singularity of $d^s_L$ at the boundary $t=0$ is removed. 
Moreover, we remind that in dimension one, the semi-norms $TV$ and $TV_{\ell^1}$ are equivalent.\BLK
\begin{proposition}\label{one_d_cuoweixiangjian}
	Let $w\in BV^s(I)\cap C^\infty(I)$ be given. Let $TV^s(w,I)$ denote the $s$-order total variation of $w$ in $I=(0,1)$ and
	\be
	\tilde w(x):=
	\begin{cases}
		w(x)&\text{ for }x\in I,\\
		0&\text{ if }x\in \R\setminus I.
	\end{cases}
	\ee
	Then we have
	\be\label{dengchaxiangjian_p1}
	\norm{\tau_h\tilde w-\tilde w}_{L^1(\R)}\leq h^s \fsp{TV^s(w)+\norm{w}_{L^\infty(\partial I)}}
	{  +O(|h|) },
	\ee
	provided that $\norm{w}_{L^\infty(\partial I)}<+\infty$.
\end{proposition}

\begin{proof} 
	We start by recalling the fractional mean value formula for $0<s<1$ from \cite[Corollary 4.3]{MR2237634}:  for $h\in\R$ small,
	\be\label{revised_mean_value}
	w(x+h)=w(x)+h^s\frac{1}{\Gamma(1+s)}\hat d_L^sw(x+\theta h),
	\ee
	where the revised left-sided R-L $s$-order derivative $\hat d^s_L$ is defined in \eqref{revisied_RL} and $\theta\in\R$, say $\theta(h)$, depends upon $h$, and satisfies
	\be
	\lim_{h\to 0}\theta^s(h)=\frac{\Gamma(1+s)^2}{\Gamma(1+2s)}.
	\ee
	We note that 
	\begin{align}
		\hat d^s_Lw(x)&=\frac1{\Gamma(1-s)} \frac d{dx}\int_0^x \frac{w(t)-w(0)}{\fsp{x-t}^s}dt\notag \\
		&= \frac1{\Gamma(1-s)} \frac d{dx}\int_0^x \frac{w(t)}{\fsp{x-t}^s}dt - \frac1{\Gamma(1-s)} \frac d{dx}\int_0^x \frac{w(0)}{\fsp{x-t}^s}dt \notag \\
		&=d^s_L w(x)-w(0)\frac1{\Gamma(1-s)}\frac1{x^s}.\label{revised_mean_value2}
	\end{align}
	We also, by the definition of $\tilde w$, summarize the following 4 cases (w.l.o.g we only consider $\abs{h}<0.1$).
	\begin{enumerate}[{Case} 1.]
		\item Both $x$ and $(x+h)\in\R\setminus I$. In this case we have
		\be
		\abs{\tilde w(x+h)-\tilde w(x)}=0;
		\ee
		\item
		$x\in I$, $h<0$ and $x+h<0$. In this case we have $0<x<-h$ and
		\be
		\abs{\tilde w(x+h)-\tilde w(x)}\leq \abs{\tilde w(x+h)- w(0)}+\abs{w(0)- w(x)}=\abs{w(0)}+\abs{w(0)- w(x)};
		\ee
		\item
		$x\in I$, $h>0$ and $x+h>1$. In this case we have $0<1-x<h$ and
		\be
		\abs{\tilde w(x+h)-\tilde w(x)}\leq \abs{\tilde w(x+h)- w(1)}+\abs{w(1)- w(x)}=\abs{w(1)}+\abs{w(1)- w(x)};
		\ee
		\item
		Both $x$ and $x+h\in I$. In this case we have 
		\be
		\abs{\tilde w(x+h)-\tilde w(x)}=\abs{ w(x+h)- w(x)}.
		\ee
	\end{enumerate}
	We now claim \eqref{dengchaxiangjian_p1}. 
	In any of above cases, we could deduce that 
	\be\label{revised_mean_value3}
	\norm{\tilde \tau_hw-\tilde w}_{L^1(\R)}\leq \abs{I^h\setminus I_h}\norm{w}_{L^\infty(\R)}+\norm{ \tau_hw- w}_{L^1(I_h)},
	\ee
	where $I_h=(\abs{h},1-\abs{h})$ and $I^h=(-\abs{h},1+\abs{h})$.
	{  Clearly, the first term $ \abs{I^h\setminus I_h}\norm{w}_{L^\infty(\R)}$ is of order $O(|h|)$.}
	Next, for $x$ and $x+h\in I$, we observe that
	\begin{align*}
		\abs{\tilde w(x+h)-\tilde w(x)}&\leq h^s\frac{1}{\Gamma(1+s)}\abs{\hat d_L^sw(x+\theta h)}\\
		&\leq h^s\frac{1}{\Gamma(1+s)}\fmp{\abs{d^s_L w(x)}+w(0)\frac1{\Gamma(1-s)}\frac1{x^s}},
	\end{align*}
	where we used \eqref{revised_mean_value} and \eqref{revised_mean_value2}.\\\\
	That is, we have
	\be
	\norm{\tilde w(x+h)-\tilde w(x)}_{L^1(I_h)}\leq h^s\frac{1}{\Gamma(1+s)}\fmp{TV^s(w,I)+\norm{w}_{L^\infty(\partial I)}[\Gamma(1-s)(1-s)]^{-1}},
	\ee
	and combining with \eqref{revised_mean_value3} we obtain 
	\be
	\norm{\tilde \tau_hw-\tilde w}_{L^1(\R)}\leq h^s\fmp{[\Gamma(1-s)(1-s)]^{-1}+[\Gamma(1+s)]^{-1}}\fmp{TV^s(w)+\norm{w}_{L^\infty(\partial I)}}{  +O(|h|) },
	\ee
	hence \eqref{dengchaxiangjian_p1}.
\end{proof}

\begin{proof}(of Theorem \ref{ATV_real_embedding})
	The main idea of the proof is to combine an approximation and a slicing argument.
	We only write the proof for the case $N=2$, as the general case $N\geq 3$ can be obtained similarly. Let $\seqn{u_n}\subset BV^s(Q)$ be given, such that the assumption \eqref{image_condition_infty} holds. {  Define}
	\be
	\tilde u_n(x):=
	\begin{cases}
		u_n(x) &\text{ if }x\in Q,\\
		0 &\text{ if }x\in \R^2\setminus Q.
	\end{cases}
	\ee
	Then, by Proposition \ref{one_d_cuoweixiangjian}, we have, for any given $x_2\in \R$ and small $\varepsilon$,
	\be
	\norm{\tau_\varepsilon \tilde u_n( {  x_1 } ,x_2)-\tilde u_n({  x_1 },x_2)}_{L^1(\R)}\leq \varepsilon^s C_s [TV_{\ell^1}(u_n({  x_1 },x_2),I)+\norm{u_n}_{L^\infty(\partial Q)}]   {  +O(|\varepsilon|) },
	\ee
	where $C_s:=\fmp{[\Gamma(1-s)(1-s)]^{-1}+[\Gamma(1+s)]^{-1}}$. Thus, we have, for vector  $h\in \R^2$ with small
	norm and parallel to the $x_1$ axis, 
	\be
	\norm{\tau_h \tilde u_n-\tilde u_n}_{L^1(\R^2)}\leq h^s C_s \fmp{TV_{\ell^1}^s(u_n)+\norm{u_n}_{L^\infty(\partial Q)}}{  +O(|h|) }.
	\ee
	Similarly, the above estimate holds for $h\in \R^2$ with small norm and parallel to the $x_2$ axis.
	For a general direction vector $h\in\R^2$, we write $h=(h_1,h_2)$, and 
	\begin{align*}
		\norm{\tau_h \tilde u_n-\tilde u_n}_{L^1(\R^2)}&\leq \norm{\tau_{h_1} \tilde u_n-\tilde u_n}_{L^1(\R^2)}+\norm{\tau_{h_2} \tilde u_n-\tilde u_n}_{L^1(\R^2)}\\
		&\leq 2h^s C_s C\fmp{\norm{u_n}_{BV^s(Q)}+\norm{u_n}_{L^\infty(\partial Q)}}{  +O(|h|) }.
	\end{align*}
	Thus, in view of Theorem \ref{lp_compactness}, there exists $u\in L^1(Q)$
	such that $\tilde u_n\to u$ strongly in $L^1(Q)$. That is, $u_n\to u$ strongly in $L^1(Q)$, and the
	proof is complete.
\end{proof}

\begin{corollary}
	Given a sequence $\seqn{u_n}\subset BV^s(Q)$ such that the assumptions of Theorem \ref{ATV_real_embedding} hold, then, up to a sub-sequence, there exists $u\in BV^s(Q)$ 
	such that 
	\be
	u_n\to u\text{ strongly in }L^1(Q)\quad\text{ and }
	\quad\liminfn \,TV^s(u_n)\geq TV^s(u).
	\ee
\end{corollary}
\begin{proof}
	The proof is done by combining Theorems \ref{ATV_real_embedding} and \ref{weak_star_comp_s}.
\end{proof}
\subsection{Lower semi-continuity with respect to sequences of orders}
We first perform our analysis for the an-isotropic total variation $TV_{\ell^1}^r$, and then use the equivalence condition (Remark \ref{an_iso_equ}) 
to recover the case of isotropic total variation $TV^r$.

\subsubsection{Interpolation properties of an-isotropic total variation}\label{subsub_monotone}

We start by studying an monotonicity result of $TV_{\ell^1}^s$ with respect to the order $s\in(0,1)$, for function $u\in IM(Q)$ (recall space $IM(Q)$ from \eqref{image_function}).
\begin{proposition}[monotonicity of an-isotropic $TV_{\ell^1}^s$]\label{bergounioux2017fractional}
	Let $0< s<t<1$ be given and $u\in SV_{\ell^1}^{t}(Q)\cap\, IM(Q)$. Then we have $u\in BV^s(Q)\cap IM(Q)$ and
	\be
	TV_{\ell^1}^{s}(u)\leq TV_{\ell^1}^{t}(u).
	\ee
	In particular, at $s=0$, we have
	\be
	\norm{u}_{L^1(Q)}\leq TV_{\ell^1}^{  t}  (u).
	\ee
\end{proposition}

\begin{proof} 
	For a moment we assume that $0<s<t<1$. We deal in one dimension first and we assume also that  $w\in BV^{t}(I)\cap C^\infty(I)$. We claim $w\in \mathbb I^{t}(L^1(I))$. Indeed, we have 
	\be
	d (\mathbb I^{1-s}w(x))=\frac d{dx}\int_0^x \frac{w(t)}{(x-t)^{s}}dt = d^{s} w(x).
	\ee
	Thus, by assumption we have
	\be\label{cite_represent_frac1}
	\|\mathbb I^{1-s}w\|_{W^{1,1}(I)}=TV^{s}(w)<+\infty,
	\ee
	and hence \eqref{inte_frac_cond} holds.\\\\
	We next claim \eqref{bdy_frac_cond}. Indeed, we only need to consider the case $l=0$ and we observe that
	\be
	\abs{\mathbb I^{1-s}w(\e)}=\abs{\int_0^\e\frac{w(t)}{(\e-t)^s}dt}\leq \norm{w}_{L^\infty(\partial I)}\int_0^\e\frac{1}{(\e-t)^s}dt\leq  \norm{w}_{L^\infty(\partial I)}\e^{1-s}\to 0
	\ee
	which implies that
	\be\label{cite_represent_frac2}
	d^0 \mathbb I^{1-s}w(0)=0.
	\ee
	Thus, \eqref{cite_represent_frac1} and \eqref{cite_represent_frac2} allows us to use Theorem \ref{thm_MR1347689}, Assertion \ref{cite_represent_frac} 
	to infer the existence of $f\in L^1(I)$ such that $w=\mathbb I^{t}[f]$. Moreover, in view of Theorem \ref{thm_MR1347689}, Assertion \ref{MR1347689T2_5}, we have
	\be
	w = \mathbb I^{t}[f] = \mathbb I^{s}\mathbb I^{t-s}[f].
	\ee
	We observe that
	\be\label{point_norm_single}
	\norm{d^{s}w}_{L^1(I)} = \norm{\mathbb I^{t-s}[f]}_{L^1(I)}\\
	\leq \abs{\mathbb I^{t-s}}\norm{f}_{L^1(I)},
	\ee
	where by $\abs{\mathbb I^{t-s}}$ we denote the operator norm. Thus, in view of \eqref{eq_semigroup_frac_int},
	and setting $\delta=t-s>0$, we have
	\be
	\abs{\mathbb I^{\delta}}\leq \sup_{\norm{\vp}_{L^1}>0}\frac{\norm{\mathbb I^{\delta}\vp}_{L^1(a,b)}}{\norm{\vp}_{L^1(a,b)}}\leq (b-a)^{\delta}\frac{1}{\delta\Gamma(\delta)}=(b-a)^{\delta}\frac{1}{\Gamma(\delta+1)}\leq 1,
	\ee
	whenever $b-a\leq 1$. This, together with \eqref{point_norm_single}, gives
	\be
	\norm{d^{s}w}_{L^1(I)} \leq \norm{d^{t}w}_{L^1(I)},
	\ee
	as desired.\\\\
	We next deal with the multi-dimensional case. We shall only write in details for $N=2$, as the case $N\geq 3$ 
	is similar. Assume that $u\in BV^t(Q)\cap C^{\infty}(Q)$, and in view of Remark \ref{smooth_use_here}, we have
	\be\label{real_momn_two2}
	TV_{\ell^1}^t(u)=\int_Q\abs{\nabla^s u}dx=\int_0^1\int_0^1\abs{\partial^t_1 u(x)}dx_1dx_2+\int_0^1\int_0^1\abs{\partial^t_2 u(x)}dx_1dx_2.
	\ee
	Since $u\in C^\infty(Q)\cap BV^t(Q)$, we have, for each $x_2\in(0,1)$,
	the slice $w_{x_2}(x_1):=u(x_1,x_2)$ is well defined and belongs to $BV^t(I)$. Thus
	\be
	\int_0^1\abs{\partial^s_1u(x_1,x_2)}dx_1= TV_{\ell^1}^s(w_{x_2}(x_1))\leq TV_{\ell^1}^t(w_{x_2}(x_1))=\int_0^1\abs{\partial^t_1u(x_1,x_2)}dx_1.
	\ee
	Integrating over $x_2\in I$, we have 
	\be\label{real_momn_two}
	\int_0^1\int_0^1\abs{\partial^s_1 u(x)}dx_1dx_2\leq \int_0^1\int_0^1\abs{\partial^t_1 u(x)}dx_1dx_2.
	\ee
	The same conclusion, but with $\pd_2^s$ (resp. $\pd_2^t$) instead
	of $\pd_1^s$ (resp. $\pd_1^t$), can be proven using the exact same arguments.
	This, together with \eqref{real_momn_two} and \eqref{real_momn_two2}, gives
	\be\label{smooth_result_tvts}
	TV_{\ell^1}^s(u)\leq TV_{\ell^1}^t(u)\text{ for $u\in C^\infty(Q)\cap BV^t(Q)$}.
	\ee 
	Finally, assume $u\in SV_{\ell^1}^t(Q)$ only. From Definition \ref{strict_conv_BVr} we could obtain a sequence $\seqn{u_n}\subset BV^t(Q)\cap C^\infty(Q)$ such that $u_n\to u$ strongly in $L^1(Q)$ and 
	\be\label{monon_approx}
	TV_{\ell^1}^t(u_n)\to TV_{\ell^1}^t(u).
	\ee
	Then, in view of \eqref{smooth_result_tvts}, for each $u_n$ we have 
	\be
	TV_{\ell^1}^s(u_n)\leq TV_{\ell^1}^t(u_n).
	\ee
	Also, since $u_n\to u$ strongly in $L^1(Q)$, in view of Theorem \ref{weak_star_comp_s}, and together with \eqref{monon_approx}, we conclude that 
	\be
	TV_{\ell^1}^s(u)\leq \liminfn\, TV_{\ell^1}^s(u_n)\leq \limsup_{n\to\infty} TV_{\ell^1}^t(u_n)=TV_{\ell^1}^t(u).
	\ee
	In the end, take any $\vp\in C_c^\infty(Q;\R^2)$, we observe, in view of Lemma \ref{linftyds} and \eqref{scaled_divg}, that
	\be
	\liminf_{s\searrow 0}TV_{\ell^1}^s(u)\geq \liminf_{s\searrow 0}\int_Qu\,\divg^s\vp\,dx =\frac12\int_Q u\fsp{\vp_1+\vp_2}\,dx,
	\ee
	and hence, we conclude, by the arbitrariness of $\vp$, that 
	\be
	TV_{\ell^1}^s(u)\geq TV_{\ell^1}^0(u)={  \|u\|_{L^1(Q)} },
	\ee
	which concludes the case that $s=0$. The proof is thus complete.
\end{proof}

We next investigate the lower semi-continuity and compactness with respect to the
order. 
\begin{proposition}\label{compact_lsc_s}
	Given sequences $\seqn{s_n}\subset (0,1)$, $s_n\to s\in[0,1]$, and $\seqn{u_n}\subset L^1(Q)$ such that there exists $p\in(1,+\infty]$ satisfying
	\be\label{sn_uniform_uppd}
	\sup\flp{\norm{u_n}_{L^p(Q)}+TV^{s_n}(u_n):\,\,n\in\N}<+\infty,
	\ee
	then the following statements hold.
	\begin{enumerate}[1.]
		\item
		There exists $u\in BV^s(Q)$ and, up to a sub-sequence, $u_n\wto u$ weakly in $L^p(Q)$, and 
		\be\label{sn_uniform_lsc}
		\liminf_{n\to\infty} TV^{s_n}(u_n)\geq TV^s(u).
		\ee
		\item
		Assuming in { addition} that $u_n\in SV^{s_n}(Q)$ for each $n\in\N$, $s_n\to s\in(0,1]$, and $\seqn{u_n}\subset IM(Q)$, we have 
		\be
		u_n\to u\text{ strongly in }L^1(Q).
		\ee
	\end{enumerate}
\end{proposition}

\begin{proof}
	Let $\seqn{u_n}\subset L^1(Q)$ be given, such that \eqref{sn_uniform_uppd} holds. Since $\norm{u_n}_{L^p(Q)}$ (with $p>1$) is uniformly bounded, 
	there exists $u\in L^p(Q)$ such that, upon sub-sequence, 
	\be\label{need_dunford_pettis}
	u_n\wto u\text{ weakly in }L^p(Q).
	\ee
	We now prove Statement 1. In view of Lemma \ref{linftyds} we get that for any $\vp\in C_c^\infty(Q)$ and $s\in[0,1]$, it holds $\divg^s\vp\in L^{p'}(Q)$ (where $1/p +1/{p'}=1$) 
	and hence, by the weak $L^p(Q)$-convergence in \eqref{need_dunford_pettis}, we have that 
	\be
	\limn \int_Q u_n\,\divg^s\vp\,dx= \int_Q u\,\divg^s\vp\,dx.
	\ee
	Statement 2 of Lemma \ref{linftyds} gives $\divg^{s_n}\vp\to \divg^s\vp$ a.e., while
	Statement 1 of Lemma \ref{linftyds} gives
	\[\sup_n \|\divg^{s_n}\vp\|_{L^\infty(Q)}<+\infty.\]
	Thus, by dominated convergence theorem, $\divg^{s_n}\vp\to \divg^s\vp$ strongly
	in $L^{p'}(Q)$.
	By H\"older inequality, 
	\begin{align*}
		\limsup_{n\to\infty}\int_Q \abs{u_n}\abs{\divg^{s_n}\vp-\divg^s\vp}dx
		\leq \sup_{n\in\N}\norm{u_n}_{L^p(Q)} \limsup_{n\to\infty} \norm{\divg^{s_n}\vp-\divg^s\vp}_{L^{p'}(Q)}=0.
	\end{align*}
	Therefore, we have
	\begin{align}
		\limn &\int_Q u_n \,\divg^{s_n}\vp\,dx\notag\\
		&\geq\liminfn \int_Q u_n\,\divg^{s}\vp\,dx+\liminfn \int_Q u_n\fmp{\divg^{s_n}\vp-\divg^s\vp}dx\notag\\
		&=\int_Q u\,\divg^s\vp\,dx,
		\label{need_for_tv_case}
	\end{align}
	and hence
	\be
	\liminfn\, TV^{s_n}(u_n)\geq \liminfn\,\int_Q u_n\,\divg^{s_n}\vp\,dx 
	=\limn \int_Q u_n\,\divg^{s_n}\vp\,dx= \int_Q u\,\divg^s\vp\,dx,
	\ee
	and \eqref{sn_uniform_lsc} follows by the arbitrariness of $\vp\in C_c^\infty(Q;\rn)$.\\\\
	We next prove Statement 2. By Remark \ref{an_iso_equ} and \eqref{sn_uniform_uppd} we have that 
	\be
	\sup\flp{\norm{u_n}_{L^p(Q)}+TV_{\ell^1}^{s_n}(u_n):\,\,n\in\N}<+\infty,
	\ee
	Since $s_n\to s\in(0,1]$, in view of Proposition \ref{bergounioux2017fractional}, 
	there exist a (sufficiently large) ${  n_0}\in\N$ and $s_{  n_0}<1$, such that 
	\be
	TV_{\ell^1}^{s_{  n_0}/2}(u_n)\leq TV_{\ell^1}^{s_n}(u_n)+1.\qquad \text{for all }n\ge {  n_0}.
	\ee
	Combined with the assumption that $\seqn{u_n}\subset IM(Q)$ allows us to use Theorem \ref{ATV_real_embedding} to infer Statement 2.
\end{proof}
The next theorem provides a connection between the lower order $TV^s$ to the higher order $TV^r$, where we recall again that $r=\ir+s$.
\begin{theorem}\label{liangbiankongzhizhongjian}
	Let $r=\ir+s$ be given. There exists a constant $C_r>0$ such that 
	\be
	TV^{s}(u)\leq C_r\fmp{\norm{u}_{L^1(Q)}+TV^{r}(u)},
	\ee
	for all $u\in SV^r(Q)$.
\end{theorem}
We will first show an one-dimensional version. 

\begin{lemma}\label{1d_approx_up}
	Let $I=(0,1)$ and $r=\ir+s$ be given. There exists a constant $C_r>0$ such that 
	\be
	\sum_{l=0}^{\ir-1}TV^{s+l}(w)\leq C_r\fmp{\norm{w}_{L^1(I)}+TV^{r}(w)},
	\ee
	for all $w\in SV^r(I)$.
\end{lemma}

\begin{proof} 
	We deal with the case $\ir=1$ first. That is, we have $r=1+s$.\\\\
	Assume no such $C_r$ exists, i.e., there exists a sequence $\seqn{w_n}\subset SV^{1+s}$ such that, for each $n\in\N$,
	\be
	TV^s(w_n)=1\text{ and }\norm{w_n}_{L^1(I)}+TV^{1+s}(w_n)<1/n.
	\ee
	In view of Theorem \ref{approx_smooth}, we may as well assume that $\seqn{w_n}\subset C^\infty(I)\cap BV^{1+s}(I)$ and we may write
	\be\label{where_contra}
	\norm{d^sw_n}_{L^1(I)}\geq1/2\quad\text{ and }\quad\norm{w_n}_{L^1(I)}+\norm{d^{1+s}w_n}_{L^1(I)}<2/n.
	\ee
	Thus, we have that $\seqn{w_n}\subset W^{1+s}(I)$, and 
	\be\label{strong_zero}
	w_n\to 0\quad\text{ and }\quad d^{1+s}w_n\to 0\quad \text{ strongly in }L^1(I).
	\ee
	Recall from Theorem \ref{thm_MR3144452}, we may write that, for each $w_n$, 
	\be
	w_n(t)= \frac{c_{0,n}}{\Gamma(s)}t^{s-1}+ \frac{c_{1,n}}{\Gamma(s+1)}t^{s}+\mathbb I^{1+s} \phi_n(t),\,\, t\in I\,\,a.e.,
	\ee
	and 
	\be
	d^{1+s} w_n(t)=\phi_n(t).
	\ee
	Thus, in view of \eqref{strong_zero} we have 
	\be
	\phi_n\to 0\text{ strongly in }L^1(I),
	\ee
	and together with \eqref{eq_semigroup_frac_int}, we have, for any $0\leq s'\leq s$,
	\be\label{shangyidian0}
	\norm{\mathbb I^{1+s'}\phi_n(t)}_{L^1(I)}\leq \frac1{\Gamma(2+s')}\norm{\phi_n}_{L^1(I)}\to 0.
	\ee
	We next claim that 
	\be\label{shangyidian1}
	c_{0,n}\to 0\text{ and }c_{1,n}\to 0.
	\ee
	By the mean value theorem we have
	\be
	(\mathbb I^{1-s}w)(t_0)=\int_I (\mathbb I^{1-s}w)(l)dl.
	\ee
	Since $w\in W^{1+s}(I)$, we have that $(\mathbb I^{1-s}w)$ is absolutely continuous, hence
	\be
	(\mathbb I^{1-s}w)(t)=(\mathbb I^{1-s}w)(t_0)+\int_{t_0}^t d^1(\mathbb I^{1-s}w)(l)dl = (\mathbb I^{1-s}w)(t_0)+\int_{t_0}^t (d^sw)(l)dl.
	\ee
	Thus, for any $t\in[0,1]$,
	\be
	\abs{(\mathbb I^{1-s}w)(t)}\leq \norm{\mathbb I^{1-s}w}_{L^1(I)}+\norm{d^s u}_{L^1}\leq \frac{1}{\Gamma(2-s)}\norm{w}_{L^1(I)}+\norm{d^s w}_{L^1(I)},
	\ee
	where at the last inequality we used \eqref{eq_semigroup_frac_int}. This, and together with \eqref{d_representable}, gives that 
	\be
	\abs{c_{0,n}}=\abs{(\mathbb I^{1-s}w_n)(0)}\leq \frac{1}{\Gamma(2-s)}\norm{w_n}_{L^1(I)}+\norm{d^s w_n}_{L^1(I)}.
	\ee
	Similarly, we may show that 
	\be
	\abs{c_{1,n}}=\abs{(d(\mathbb I^{1-s})w_n)(0)}\leq \norm{d^sw_n}_{L^1(I)}+\norm{d^{1+s} w_n}_{L^1(I)}.
	\ee
	Thus, in view of \eqref{where_contra}, we have 
	\be\label{const_upper_bdd}
	\sup\flp{\abs{c_{0,n}}+\abs{c_{1,n}}:\,\, n\in\N}<+\infty.
	\ee
	We also notice that 
	\begin{align}
		\norm{\frac{c_{0,n}}{\Gamma(s)}t^{s-1}+ \frac{c_{1,n}}{\Gamma(s+1)}t^{s}}_{L^1(I)} &=\norm{\frac{c_{0,n}}{\Gamma(s)}t^{s-1}+ \frac{c_{1,n}}{\Gamma(s+1)}t^{s}+\mathbb I^{1+s}\phi_n(t)-\mathbb I^{1+s}\phi_n(t)}_{L^1(I)}\notag\\
		&=\norm{w_n-\mathbb I^{1+s}\phi_n(t)}_{L^1(I)}\notag\\
		&\leq \norm{w_n}_{L^1(I)}+\norm{\mathbb I^{1+s}\phi_n(t)}_{L^1(I)}\to 0,
		\label{jiajiajia_l1}
	\end{align}
	which, combined with \eqref{const_upper_bdd}, implies 
	\[c_{0,n}\to 0\text{ and }c_{1,n}\to 0,\]
	and hence \eqref{shangyidian1}.
	Then, in view of Lemma \ref{power_function_s}, we have
	\be
	d^s t^{s-1}=0 \text{ and }d^s t^s = \Gamma(s+1).
	\ee
	Next, in view of Definition \ref{frac_represent_I} and Theorem \ref{thm_MR1347689}, Assertion \ref{cite_represent_frac}, 
	we have $\phi\in \mathbb I^s(L^1(I))$. Combined with  Statement 2 of Theorem \ref{thm_MR1347689}, Assertion \ref{MR1347689T2_5}, we obtain 
	\be
	d^s \mathbb I^{1+s}\phi_n(t)=\mathbb I^1\phi_n(t).
	\ee
	Thus, we have
	\be
	d^s w_n = c_{1,n}+\mathbb I^1\phi_n(t),
	\ee
	and together with \eqref{shangyidian0} and \eqref{shangyidian1}, we conclude that 
	\be\label{shangyidian2}
	\limsup_{n\to 0}\norm{d^{s} w_n}_{L^1(I)}\leq \abs{c_{1,n}}+\norm{\mathbb I^1\phi_n}_{L^1(I)}\to 0,
	\ee
	which contradicts \eqref{where_contra}.\\\\
	Now we consider the general case $\ir\in\N$. In view of \eqref{repre_AC}, we write 
	\be
	w_n(t)=\sum_{l=0}^\ir \frac{c_{l,n}}{\Gamma(s+l)}t^{s-1+l}+\mathbb I^r\phi_n(t),\,\, t\in I\,\,a.e..
	\ee
	Similarly to \eqref{jiajiajia_l1}, we may show that  
	\be
	\abs{c_{0,n}}\leq \frac{1}{\Gamma(2-s)}\norm{w_n}_{L^1(I)}+\norm{d^s w_n}_{L^1(I)},
	\ee
	and
	\be
	\abs{c_{l,n}}\leq \norm{d^{s+l-1}w_n}_{L^1(I)}+\norm{d^{s+l} w_n}_{L^1(I)},
	\ee
	for $l=1,\ldots, \ir-1$, and further deduce that
	\be
	c_{l,n}\to 0\text{ for }i=0,\ldots, \ir.
	\ee
	Note that, for each $i=0,\ldots, \ir-1$, we have
	\begin{enumerate}[1.]
		\item
		$d^{l+s} t^{s-1+j}=0$, for $j=0,\ldots, l$;
		\item
		$d^{l+s} t^{l+s} = \Gamma(s+1)$;
		\item
		$d^{l+s} t^{l+s+j} = \frac{\Gamma(l+s+j+1)}{\Gamma(j+1)}t^j$, for $j=1,\ldots,\ir-1-l$,
		\item
		$d^{l+s}(\mathbb I^r[\phi_n](t))=\mathbb I^{\ir-l}[\phi_n](t)$, a.e., $t\in I$.
	\end{enumerate}
	Therefore, we obtain that 
	\be
	\limsup_{n\to 0}\norm{d^{l+s} w_n}_{L^1(I)}\leq \sum_{j=0}^{\ir-1-l}\abs{c_{l+s+j,n}}+\norm{\mathbb I^{\ir-l}[\phi_n]}_{L^1(I)}\to 0,
	\ee
	which is a contradiction, and hence we conclude the proof.
\end{proof}

\begin{proof}(of Theorem \ref{liangbiankongzhizhongjian})
	We only deal with the case $N=2$, as the case $N\geq 3$ is similar. We first show that 
	\be\label{new_ATV_conclude}
	TV_{\ell^1}^{s}(u)\leq C_r\fmp{\norm{u}_{L^1(Q)}+TV_{\ell^1}^{r}(u)}.
	\ee
	We assume for a moment that $u\in BV^r(Q)\cap C^\infty(Q)$. We start with the case $\ir=1$. That is, $r=1+s$. In view of Remark \ref{smooth_use_here}, we have
	\be\label{eq_s_order}
	TV_{\ell^1}^s(u)=\int_Q\abs{\partial_1^su}dx+\int_Q\abs{\partial_2^su}dx
	\ee
	and
	\be\label{eq_1s_order}
	TV_{\ell^1}^{1+s}(u)=\int_Q\abs{\partial_1^{1+s}u}dx+\int_Q\abs{\partial_2^{1+s}u}dx+\int_Q\abs{\partial_1^{s}\partial_2u}dx+\int_Q\abs{\partial_2^{s}\partial_1u}dx.
	\ee
	Let $w(t):=u(t,x_2)$, with a fixed $x_2\in I$. Then, by Lemma \ref{1d_approx_up}, we have
	\be
	TV_{\ell^1}^s( { w} )\leq C_r\fmp{\norm{{ w}}_{L^1(I)}+TV_{\ell^1}^{1+s}({ w})}. 
	\ee
	That is, 
	\be
	\int_0^1\abs{\partial_1^su(x_1,x_2)}dx_1\leq C_r\fmp{\int_0^1 \abs{u(x_1,x_2)}dx_1+\int_0^1\abs{\partial_1^{1+s}u(x_1,x_2)}dx_1},
	\ee
	and hence
	\begin{align*}
		\int_Q\abs{\partial_1^su}dx &= \int_0^1\int_0^1\abs{\partial_1^su(x_1,x_2)}dx_1dx_2\\
		&\leq C_r\fmp{\int_0^1\int_0^1 \abs{u(x_1,x_2)}dx_1dx_2+\int_0^1\int_0^1\abs{\partial_1^{1+s}u(x_1,x_2)}dx_1dx_2} \\
		&=C_r\fmp{ \norm{u}_{L^1(Q)}+\int_Q \abs{\partial_1^{1+s}u}dx}.
	\end{align*}
	We may analogously prove 
	\be
	\int_Q\abs{\partial_2^su}dx \leq C_r\fmp{ \norm{u}_{L^1(Q)}+\int_Q \abs{\partial_2^{1+s}u}dx},
	\ee 
	and, in view of \eqref{eq_s_order} and \eqref{eq_1s_order},
	\be\label{eq_1s_order_smooth}
	TV_{\ell^1}^s(u)\leq C_r\fmp{\norm{u}_{L^1(Q)}+TV_{\ell^1}^{1+s}(u)},\text{ for each $u\in C^\infty(Q)\cap BV^r(Q)$. }
	\ee
	To conclude, we take an approximating sequence $\seqn{u_n}\subset C^\infty(Q)\cap BV^{1+s}(Q)$ such that $u_n\to u$ strongly in 
	$L^1(Q)$ and $TV_{\ell^1}^{1+s}(u_n)\to TV_{\ell^1}^{1+s}(u)$. The former implies that 
	\be
	\liminfn TV_{\ell^1}^s(u_n)\geq TV_{\ell^1}^s(u),
	\ee
	and together with \eqref{eq_1s_order_smooth} we conclude that, for $u\in SV^{1+s}(Q)$,
	\begin{multline}
		TV_{\ell^1}^s(u)\leq \liminfn TV_{\ell^1}^s(u_n)\leq \liminf_{n\to\infty}2C_r\fmp{\norm{u_n}_{L^1(Q)}+TV_{\ell^1}^{1+s}(u_n)} \\
		\leq \limsup_{n\to\infty}2C_r\fmp{\norm{u_n}_{L^1(Q)}+TV_{\ell^1}^{1+s}(u_n)} = 2C_r\fmp{\norm{u}_{L^1(Q)}+TV_{\ell^1}^{1+s}(u)},
	\end{multline}
	as desired.\\\\
	Now we assume $\ir=2$. Similarly to the case $\ir=1$, we observe that 
	\be
	\int_Q\abs{\partial_j^su}dx \leq C_r\fmp{ \norm{u}_{L^1(Q)}+\int_Q \abs{\partial_j^{2+s}u}dx},\qquad j=1,2.
	\ee
	That is, we conclude that
	\begin{align*}
		&TV_{\ell^1}^s(u)\leq 2C_r\fmp{ \norm{u}_{L^1(Q)}+\int_Q \abs{\partial_1^{2+s}u}dx+\int_Q \abs{\partial_2^{2+s}u}dx}\\
		&\leq 2C_r\fmp{ \norm{u}_{L^1(Q)}+\int_Q \abs{\partial_1^{2+s}u}dx+\int_Q \abs{\partial_2^{2+s}u}dx+\int_Q \abs{\partial_1^{1+s}\partial_2u}dx+\int_Q \abs{\partial_2^{1+s}\partial_1u}dx}\\
		&=2C_r\fmp{\norm{u}_{L^1(Q)}+TV_{\ell^1}^r(u)}.
	\end{align*}
	For the general case that $\ir\in \N$, we shall always have, in view of Lemma \ref{1d_approx_up}, that 
	\be
	TV_{\ell^1}^s(u)\leq 2C_r\fmp{ \norm{u}_{L^1(Q)}+\int_Q \abs{\partial_1^{\ir+s}u}dx+\int_Q \abs{\partial_2^{\ir+s}u}dx},
	\ee
	and we conclude \eqref{new_ATV_conclude} as the right hand side is bounded by $\norm{u}_{L^1(Q)}+TV_{\ell^1}^r(u)$. Finally, we conclude our thesis by invoking again 
	condition \eqref{equ_p_1_p}.
\end{proof}

We close sub-section \ref{subsub_monotone} by proving that the constant $C_r$ from
Theorem \ref{liangbiankongzhizhongjian} can be taken independently of $r$. 
\begin{proposition}\label{zhendexiao_le}
	Let { $r=\lfloor r\rfloor +s$, $s\in (0,1)$,} be given. Then there exists $C>0$, {  independent of $r$}, such that
	\be
	TV^{s}(u)\leq C\fmp{\norm{u}_{L^1(Q)}+TV^{{  r}}(u)}
	\ee
	for all $u\in SV^{{  r}}(Q)$.
\end{proposition}
\begin{proof}
	We only proof this proposition for case of dimension one, i.e., $N=1$. The case in which $N\geq 2$ can be obtained from
	the one-dimensional result, and the arguments from Theorem \ref{liangbiankongzhizhongjian}. \\\\
	Let $w\in BV^{{  r}}(I)$ be given. In view of Theorem \ref{liangbiankongzhizhongjian} we have, for each ${  r}$, a constant $
	{  C_r}>0$
	(depending on {  $r$}) such that 
	\be\label{unif_upper_bdd}
	\sum_{l=0}^{\ir-1}TV^{s+l}(w)\leq C_{  r}\fmp{\norm{w}_{L^1(I)}+TV^{{  r}}(w)}
	\ee
	for all $w\in L^1(I)$. \\\\
	We shall only deal with the case $\ir=1$, as the case $\ir>1$ can be treated analogously. Suppose \eqref{unif_upper_bdd} fails, i.e.
	there exist sequences $\seqn{w_n}\subset L^1(I)$ and $\seqn{s_n}\subset (0,1)$ such that 
	\be
	TV^{s_n}(w_n)=1 \qquad\text{ and }\qquad\norm{w_n}_{L^1(I)}+TV^{1+s_n}(w_n)<1/n.
	\ee
	In view of \eqref{const_upper_bdd}, we have
	\begin{align}
		\abs{c_{0,n}}+\abs{c_{1,n}}&\leq \fmp{\frac1{\Gamma(2-s_n)}+1}\fmp{\norm{w_n}_{L^1(I)}+TV^{s_n}(w_n)+TV^{1+s_n}(w_n)}\notag\\
		&\leq \frac1{\Gamma(2-s_n)}+1+\frac1n\leq 2+\frac1n.
	\end{align}
	Hence, there exist $c_0$ and $c_1$ such that $c_{0,n}\to c_0$ and $c_{1,n}\to c_1$. Then, we may reach the contradiction by using the same arguments from \eqref{shangyidian1} to \eqref{shangyidian2}.
\end{proof}
\subsubsection{Compact embedding and lower semi-continuity}\label{subsubsec_compact}
We start again with a result on the an-isotropic total variation. Recall the image space $IM(Q)$ from \eqref{image_function}.

\begin{proposition}\label{compact_lsc_r}
	Given sequences $\seqn{r_n}\subset \R^+$ and $\seqn{u_n}\subset L^1(Q)$ such that $r_n\to r\in\R^+$ and, for some $p>1$,
	\be\label{rn_uniform_uppd}
	\sup\flp{\norm{u_n}_{L^p(Q)}+TV_{\ell^1}^{r_n}(u_n):\,\,n\in\N}<+\infty,
	\ee
	then, the following statement hold.
	\begin{enumerate}[1.]
		\item
		There exists $u\in BV^r(Q)$, such that, up to a sub-sequence, $u_n\wto u$ in $L^p(Q)$ and 
		\be
		\liminf_{n\to\infty} TV_{\ell^1}^{r_n}(u_n)\geq TV_{\ell^1}^r(u).
		\ee
		\item
		Assuming in addition that ${u_n}\subset IM(Q)\cap SV^{r_n}(Q)$ and $r_n\to r>0$, then
		\be
		u_n\to u\text{ strongly in }L^1(Q).
		\ee
	\end{enumerate}
\end{proposition}

\begin{proof} 
	Write $r_n=\lfloor r_n\rfloor+s_n$ for $s_n\in[0,1)$. By \eqref{rn_uniform_uppd}, there exists $u\in L^p(Q)$ such that, up to a sub-sequence,
	\be\label{weak_r_lp}
	u_n\wto u\text{ weakly in }L^p(Q).
	\ee
	Then Statement 1 can be proved by using the same arguments from Proposition \ref{compact_lsc_s}. \\\\
	We next prove Statement 2. We only study the case of $1\leq r_n\leq 2$, or equivalently $\ir=1$, as the case in which $\ir\geq 2$ can be dealt analogously. Again by applying Proposition \ref{zhendexiao_le} we have 
	\begin{align}
		\sup_n\norm{u_n}_{BV^{s_n}(Q)} &\leq C\sup_n\norm{u_n}_{BV^{r_n}(Q)}
		\leq \sup_n\{\norm{u_n}_{L^p(Q)}+TV_{\ell^1}^{r_n}(u_n)\}=:M<+\infty,
		\label{jiangweigongji}
	\end{align}
	where $C>0$, obtained from Proposition \ref{zhendexiao_le}, is a constant independent of $r_n$. Assume in addition that there exists $\e>0$ such that $\seqn{s_n}\subset[\e,1]$. Then, in view of Proposition \ref{compact_lsc_s}, there exists $u\in L^1(Q)$ 
	such that, up to a sub-sequence, that $u_n\to u$ strongly in $L^1(Q)$, which gives Statement 2.\\\\
	We now deal with the situation that $s_n\searrow 0$, that is, $r_n\searrow 1$. In this case, although \eqref{jiangweigongji} still holds, Proposition \ref{compact_lsc_s} does not produce a sub-sequence
	strongly converging in $L^1(Q)$. We proceed by using Theorem \ref{approx_smooth} to relax $u_n$, for each $n\in\N$, such that $u_n\in BV^{1+s_n}(Q)\cap C^\infty(Q)$. Hence, we have, for arbitrary $\vp\in C_c^\infty(Q,\R^2)$, that (recall \eqref{RLFODr})
	\be
	-\int_Q \nabla^{s_n} u_n\,\divg \vp\,dx=\int_Q u_n[\divg^{s_n}\divg] \vp\,dx=\int_Q u_n\divg^{1+s_n}\vp\,dx\leq TV_{\ell^1}^{1+s_n}(u_n),
	\ee
	which implies 
	\be\label{ds_tv_eq}
	TV_{\ell^1}(\nabla^{s_n}u_n)\leq TV_{\ell^1}^{1+s_n}(u_n)\leq M<+\infty.
	\ee
	Moreover, by Proposition \ref{zhendexiao_le}, 
	\be\label{ds_tv_eq2}
	\norm{\nabla^{s_n}u_n}_{L^1(Q)} = TV_{\ell^1}^{s_n}(u_n)\leq C \norm{u_n}_{BV^{1+s_n}(Q)}\leq M<+\infty.
	\ee
	We claim 
	\be\label{un_dsn_go}
	\limsup_{n\to\infty}\norm{u_n-\nabla^{s_n}u_n}_{L^1(Q)}=0.
	\ee
	We start from the { one-dimensional} case, i.e. $Q=I=(0,1)$, and use $w_n$ to represent $u_n$. 
	By Theorem \ref{thm_MR3144452}, for each $n\in\N$, there exists $\phi_n(t)\in L^1(I)$ such that
	\be
	w_n(t)= \frac{c_{0,n}}{\Gamma(s_n)}t^{s-1}+ \frac{c_{1,n}}{\Gamma(s_n+1)}t^{s_n}+\mathbb I^{1+s_n} \phi_n(t)
	\ee
	for a.e. $t$, and 
	\be
	d^{s_n} w_n(t) = c_{1,n} + \mathbb I^1\phi_n(t).
	\ee
	Next, in view of \eqref{r_int_frac_def}, we have, for each $n\in\N$ fixed, that 
	\be
	\mathbb I^{1+s}\phi_n(t) = \frac{1}{\Gamma(1+s)}\int_0^t \frac{\phi_n(z)}{\fsp{t-z}^{1-(1+s)}}dz=\frac{1}{\Gamma(1+s)}\int_0^t {\phi_n(z)}\fsp{t-z}^s dz.
	\ee
	That is, we have
	\be
	\lim_{t\to 0}\abs{\mathbb I^{1+s}\phi_n(t) }<+\infty\text{, for each }n\in\N.
	\ee
	However, $t^{s-1}\to +\infty$ as $t\to 0^+$. Hence, $c_{0,n}=0$ must hold, 
	as the opposite gives $w_n(t)\to+\infty$ as $t\to 0$, which contradicts our assumption $w_n\in IM(Q)$. Thus, we have
	\be
	w_n(t)= \frac{c_{1,n}}{\Gamma(s+1)}t^{s}+\mathbb I^{1+s} \phi_n(t),\qquad \text{for a.e. }t.
	\ee
	Then, direct computation gives
	\begin{align}\label{upper_un_dsn}
		\begin{split}
			\norm{w_n-d^{s_n}w_n}_{L^1(Q)} &= \norm{\frac{c_{1,n}}{\Gamma(s_n+1)}t^{s_n}-c_{1,n}+\mathbb I^{1+s_n}\phi_n-\mathbb I^1\phi_n}_{L^1(Q)}\\
			&\leq \norm{\frac{c_{1,n}}{\Gamma(s_n+1)}t^{s_n}-c_{1,n}}_{L^1(Q)}+\norm{\mathbb I^{1+s_n}\phi_n-\mathbb I^1\phi_n}_{L^1(Q)}\\
			&\leq \abs{c_{1,n}}\abs{\frac1{\Gamma(s_n+1)(s_n+1)}-1}+\abs{\mathbb I^{1+s_n}-\mathbb I^1}\norm{\phi_n}_{L^1(Q)}.
		\end{split}
	\end{align}
	Moreover, using the same argument from the proof of \eqref{const_upper_bdd}, we have 
	\be
	\abs{c_{1,n}}+\norm{\phi_n}_{L^1(Q)}\leq 2\fmp{\norm{d^{s_n}w_n}_{L^1(Q)}+\norm{d^{s+n+1}w_n}_{L^1(Q)}}.
	\ee
	This, combined with \eqref{upper_un_dsn}, implies that 
	\begin{align*}
		&\|u_n-\nabla^{s_n}u_n\|_{L^1(Q)} = \int_0^1\norm{u_n\lfloor_{x_1}-d^{s_n}u_n\lfloor_{x_1}}_{L^1(Q)}dx_1\\
		&\leq  2\fmp{\abs{\frac1{\Gamma(s_n+1)(s_n+1)}-1}+\abs{\mathbb I^{1+s_n}-\mathbb I^1}}
		\int_0^1\fmp{\norm{d^{s_n}u_n\lfloor_{x_1}}_{{ L^1(Q)}}
			+\norm{d^{s+n+1}u_n\lfloor_{x_1}}_{{ L^1(Q)}}}dx_1\\
		&\leq  2\fmp{\abs{\frac1{\Gamma(s_n+1)(s_n+1)}-1}+\abs{\mathbb I^{1+s_n}-\mathbb I^1}}\norm{u_n}_{BV^{r_n}(Q)}.
	\end{align*}
	By Theorem \ref{thm_MR1347689}, Assertion \ref{semigroup_frac_int}, we have $\abs{\mathbb I^{s_n+1}-\mathbb I^1}\to 0$, as $s_n\to 0$, and hence we conclude \eqref{un_dsn_go}. Next, by \eqref{ds_tv_eq} and \eqref{ds_tv_eq2}, and the compact embedding in standard $BV(Q)$ space, up to a sub-sequence, there exists $\bar u\in BV(Q)$, that 
	\be
	\nabla^{s_n}u_n\to \bar u\text{ strongly in }L^1(Q).
	\ee
	Hence, by \eqref{un_dsn_go}, we have 
	\be
	\norm{u_n-\bar u}_{L^1(Q)}\leq \norm{u_n- \nabla^{s_n}u_n}_{L^1(Q)}+\norm{\nabla^{s_n}u_n-\bar u}_{L^1(Q)}\to 0,
	\ee
	that is, we have $u_n\to \bar u$ strongly in $L^1(Q)$. 
	Finally, in view of \eqref{weak_r_lp}, we have $u=\bar u$, and hence the thesis.
\end{proof}

We conclude this section by proving the second main result of this article.

\begin{theorem}[lower semi-continuity and compact embedding in $TV^r_\ellp$ semi-norms]\label{compact_lsc_r_tv}
	Given sequences $\seqn{r_n}\subset \R^+$, $\seqn{p_n}\subset [1,+\infty]$, and $\seqn{u_n}\subset L^1(Q)$ such that $r_n\to r\in\R^+\cup\flp{0}$ and $p_n\to p\in[1,+\infty]$,
	and there exists $q\in(1,+\infty]$ such that 
	\be\label{uniform_tv_upperbdd}
	\sup\flp{\norm{u_n}_{L^q(Q)}+TV_{  \ell^{p_n} }^{r_n}(u_n):\,\,n\in\N}<+\infty,
	\ee
	then, the following statements hold.
	\begin{enumerate}[1.]
		\item
		There exists $u\in BV^r(Q)$ such that, up to a sub-sequence, $u_n\wto u$ weakly in $L^q(Q)$ and 
		\be
		\liminf_{n\to\infty} TV_{\ell^{p_n}}^{r_n}(u_n)\geq TV_\ellp^r(u).
		\ee
		\item
		Assuming in addition that ${u_n}\subset IM(Q)\cap SV^{r_n}(Q)$ and $r_n\to r>0$, we have 
		\be
		u_n\to u\text{ strongly in }L^1(Q).
		\ee
	\end{enumerate}
\end{theorem}
\begin{proof}
	By applying Remark \ref{an_iso_equ}, we deduce that 
	\eqref{uniform_tv_upperbdd} implies \eqref{rn_uniform_uppd}, 
	and the thesis follows by combining \eqref{need_for_tv_case} and Theorem \ref{weak_star_comp_s}.
\end{proof}

\section*{Acknowledgements}
Xin Yang Lu acknowledges the support of NSERC Grant 
{\em ``Regularity of minimizers and pattern formation in geometric minimization problems''},
and of the Startup funding, and Research Development Funding
of Lakehead University.
The authors are grateful to Todd J. Falkenholt, and the NSERC USRA grant supporting him,
for many useful comments and suggestions.

{
	\bibliographystyle{siam}
	\bibliography{RTV_theory_Part1}
}

\end{document}